\documentclass{amsart}
\usepackage[utf8]{inputenc}
\usepackage{fullpage}
\usepackage{yhmath}  
\usepackage{marvosym}
\usepackage{graphicx}
\usepackage{xcolor}
\usepackage{epsfig}
\usepackage{enumerate}
\usepackage{mathtools}
\usepackage{array}

\usepackage[frozencache=true,cachedir=.]{minted}

\usepackage{hyperref}
\hypersetup{
    colorlinks,
    linkcolor={red!50!black},
    citecolor={blue!50!black},
    urlcolor={blue!80!black}
}  %%% This tones down the bright green hyperlink box and its other colors...

\usepackage{booktabs}

\usepackage[T1]{fontenc} % makes sure _ will copy as _ and not space
\usepackage{lmodern}
\usepackage{textcomp} % provides \textquotesingle
\def\mfld#1{\texttt{\textquotesingle#1\textquotesingle}}

\usepackage{fancyvrb}
\VerbatimFootnotes

\usepackage[active]{srcltx}

\usepackage{tikz}
\usetikzlibrary{braids}
\theoremstyle{plain}
\newtheorem{theorem}{Theorem}

\newtheorem{lemma}[theorem]{Lemma}

\newtheorem{prob}[theorem]{Problem}
\newtheorem{question}[theorem]{Question}

\theoremstyle{definition}

\newtheorem{example}[theorem]{Example}

\newcommand{\calA}{\ensuremath{{\mathcal A}}}

\newcommand{\calT}{\ensuremath{{\mathcal T}}}

\newcommand{\calD}{\ensuremath{{\mathcal D}}}

\newcommand{\comment}[1]{}

\newcommand{\Z}{\ensuremath{\mathbb{Z}}}

\newcommand{\mobius}{M\"{o}bius }

\title{L-space knots with tunnel number >1 by experiment}

\author{Chris Anderson}
\author{Kenneth L.\ Baker}
\address{Department of Mathematics\\University of Miami\\ Coral Gables, FL 33146 \\ USA}
\email{canders@math.miami.edu, k.baker@math.miami.edu}

\author{Xinghua Gao}
\address{Department of Mathematics \\University of Illinois at Urbana-Champaign \\Urbana, IL 61801 \\ USA}
\email{xgao29@illinois.edu}

\author{Marc Kegel}
\address{Institut f\"ur Mathematik, Humboldt-Universit\"at zu Berlin, Unter den Linden 6, 10099 Berlin, Germany}
\email{kegemarc@math.hu-berlin.de}

\author{Khanh Le}
\address{Department of Mathematics\\Temple University\\ Philadelphia, PA 19122 \\ USA}
\email{khanh.q.le@temple.edu}

\author{Kyle Miller}
\address{Department of Mathematics\\ University of California, Berkeley\\ CA 94720-3840\\ USA}
\email{kmill@math.berkeley.edu}

\author{Sinem Onaran}
\address{Department of Mathematics\\Hacettepe University \\ Beytepe-Ankara, 06800\\ TURKEY}
\email{sonaran@hacettepe.edu.tr}

\author{Geoffrey Sangston}
\address{Department of Mathematics\\University of Maryland\\ College Park, MD 20742 \\ USA}
\email{gsangsto@umd.edu}

\author{Samuel Tripp}
\address{Department of Mathematics \\ Dartmouth College \\ Hanover, NH 03755\\USA}
\email{samuel.w.tripp.gr@dartmouth.edu}

\author{Adam Wood}
\address{Department of Mathematics and Statistics\\University of Melbourne\\ Parkville, VIC 3010 \\ Australia}
\email{awood3@student.unimelb.edu.au}

\author{Ana Wright}
\address{Department of Mathematics\\University of Nebraska-Lincoln\\ Lincoln, NE 68588 \\ USA}
\email{awright@huskers.unl.edu}

\begin{document}

\begin{abstract}
	In Dunfield's catalog of the hyperbolic manifolds in the SnapPy census which are complements of L-space knots in $S^3$, we determine that $22$ have tunnel number $2$ while the remaining all have tunnel number $1$.  Notably, these $22$ manifolds contain $9$ asymmetric L-space knot complements.  Furthermore, using SnapPy and KLO we find presentations of these $22$ knots as closures of positive braids that realize the Morton-Franks-Williams bound on braid index. The smallest of these has genus 12 and braid index 4.

\end{abstract}



\maketitle

\section{Introduction and Results}
A knot in $S^3$ with a positive Dehn surgery to a Heegaard Floer L-space \cite{Ozsvath2005} is an {\em L-space knot}.  Such knots are necessarily fibered \cite{Ni2007, Ghiggini2008} and strongly quasi-positive \cite{Hedden2010}.  Many L-space knots are also {\em braid positive},  in that they may be expressed as closures of positive braids.  For example, the torus knots that are L-space knots are the positive torus knots, and hence they are braid positive. However, there is an L-space knot known to not be braid positive.

\begin{example}
The positive trefoil $T_{2,3}$ is the $(2,3)$ torus knot. It is the only L-space knot of genus $1$ \cite{Ghiggini2008}.
Its $(2,3)$--cable $T_{2,3}^{2,3}$ is also an L-space knot by the cabling formula of \cite{Hedden2005}.  Using the cabling formula for knot genera, one finds that $T_{2,3}^{2,3}$ has genus $3$.  
As pointed out in \cite[Table 8]{Dunfield2019}, the knot $T_{2,3}^{2,3}$ is listed as the 15 crossing knot $15n124802$ in the nomenclature of \cite{Hoste1998}.

Now, for a closed positive braid diagram, Seifert's Algorithm produces a minimal genus Seifert surface for the closed braid.
Therefore if a genus $g$ knot with crossing number $c$ is the closure of a positive braid with braid index $b$ and word length $\ell\geq c$, then $2g-1 =\ell-b$.  Furthermore, we may assume that each braid generator appears in the braid word at least twice so that $\ell \geq 2(b-1)$, since otherwise there would be a smaller index positive braid whose closure is the knot.   Putting these together, one concludes that $4g \geq c$ for a knot that is the closure of a positive braid. 
Since the genus $3$ knot $T_{2,3}^{2,3}$ has crossing number greater than $12$, it cannot be braid positive.  
\end{example}   

Presumably, there are other satellite L-space knots which are not braid positive.  Nonetheless, one wonders the following.

\begin{question}[e.g.\ Problem 31.2 \cite{Hom2017}]
\label{ques:briadpos}
%   https://www.birs.ca/cmo-workshops/2017/17w5011/report17w5011.pdf
Are all hyperbolic L-space knots braid positive?
\end{question}

An affirmative answer would imply that there are finitely many hyperbolic L-space knots of any given genus, see \cite[Conjecture 6.7]{Hedden2018} and \cite[Conjecture 1.2]{Baker2015}.
Given the existence of L-space knots which are not braid positive, it seems an affirmative answer to this question is unlikely.  In attempting to find a counterexample, one may look towards hyperbolic L-space knots with other remarkable properties or remarkable origins.

Dunfield hands us such a collection of knots.\footnote{Personal communication.}   The data from \cite{Dunfield2018} determines that there are exactly 1,267 complements of knots in $S^3$ in the SnapPy census of $1$--cusped hyperbolic manifolds that can be triangulated with at most $9$ ideal tetrahedra.  Among these, the data from \cite{Dunfield2019} identifies $630$ as L-space knots, identifies $635$ as not L-space knots, and leaves $2$ of them unclassified.\footnote{These two unclassified knots have now been shown to be actually L-space knots \cite{BKM}.} Let $\calD$ be the set of these $630$ L-space knots and $2$ unclassified knots.
%Note: using data as a collective noun

A hyperbolic knot is {\em asymmetric} if the isometry group of its complement is the trivial group.
Among these $632$ knots of $\calD$,
via computations of their isometry groups,
the set 
\begin{align*}
	\calA = \{& \mfld{t12533}, \mfld{t12681}, \mfld{o9\_38928}, \mfld{o9\_39162}, \mfld{o9\_40363},\\
	& \phantom{xxxxxxxxxxxxxxxxx} \mfld{o9\_40487}, \mfld{o9\_40504}, \mfld{o9\_40582}, \mfld{o9\_42675} \}
\end{align*}
collects the $9$ asymmetric ones.  All of the rest of the knots in $\calD$ have an order $2$ symmetry group generated by a strong involution as we show in Lemma~\ref{lem:symmetry}.

For the original purpose of this note, we demonstrate that none of these $9$ knots in $\calA$ provide a negative answer to Question~\ref{ques:briadpos}.
\begin{theorem}\label{thm:asymbraidpos}
The manifolds in $\calA$ are complements of braid positive knots. \\
Braidwords for these knots are listed in Table~\ref{tab:braidwords}.
\end{theorem}

There is also a similar question about the tunnel numbers of  L-space knots.  The {\em tunnel number} of a knot (or link) is the minimum number of properly embedded arcs that need to be drilled out from its exterior to form a handlebody.   While L-space knots with arbitrarily large tunnel number may be obtained through iterated cables \cite[Proposition 23]{Baker2014},  all other known examples of hyperbolic L-space knots have tunnel number at most $2$.  In particular, see \cite[Theorem 1.13]{Motegi2016} and \cite[Proposition 10.2]{Baker2017}.
\begin{question}\label{ques:tn}
Do all hyperbolic L-space knots have tunnel number at most $2$?
\end{question}
  
Since knots with tunnel number $1$ necessarily admit an involution,  asymmetric hyperbolic knots must have tunnel number at least $2$.  So it is natural to consider the tunnel number for the asymmetric manifolds in $\calA$.  Indeed, this prompts the further exploration of whether any knot of $\calD$ provides a negative answer to Question~\ref{ques:tn}.  It turns out that none do.
\begin{theorem}\label{thm:tunnelnumber}
The knots in $\calD$ all have tunnel number at most $2$.
More specifically, there is a subset $\calT_2 \subset \calD$ of $22$ knots which contains $\calA$ so that 
\begin{itemize}
    \item the knots of $\calD-\calT_2$ have tunnel number $1$ and
    \item the knots of $\calT_2$ have tunnel number $2$.
\end{itemize}
\end{theorem}

These symmetric tunnel number $2$ knots are 
\begin{align*}
	\calT_2- \calA = \{
	& \mfld{t09284}, \mfld{t09450}, \mfld{t09633}, \mfld{t10496}, \mfld{o9\_28751}, \mfld{o9\_29751}, \mfld{o9\_32314},\\ 
	& \phantom{xxxxxxxx} \mfld{o9\_33380}, \mfld{o9\_33944}, \mfld{o9\_33959},
	  \mfld{o9\_34409}, \mfld{o9\_36380}, \mfld{o9\_40026}\},
\end{align*}
and we show in Lemma~\ref{lem:symmetry} that their only symmetry is a strong inversion.

We also extend Theorem~\ref{thm:asymbraidpos} to all knots of $\calT_2$.
\begin{theorem}\label{thm:TN2braidpos}
The manifolds in $\calT_2$ are complements of braid positive knots.\\
Braidwords for the knots of $\calT_2-\calA$ are listed in Tables~ \ref{tab:braidwordsT2} and \ref{tab:braidwordsT2-2}.
\end{theorem}

It remains to show that all the knots of $\calD-\calT_2$ are braid positive.  As our methods are computer assisted rather than purely algorithmic, we have not yet embarked upon that task.
\begin{prob}
Show that all the knots in $\calD$ are braid positive.
\end{prob}

The Morton-Franks-Williams inequality gives a lower bound on the braid index of a link in terms of its HOMFLY-PT polynomial \cite{MortonBraid, FranksWilliamsBraid}.  For the $P(v,z)$ version of the HOMFLY-PT polynomial of a link $L$, let $d_+$ and $d_-$ be the max and min degrees of $v$.  Then the Morton-Franks-Williams inequality states that 
\[(d_+-d_-)/2 + 1 \leq b\]
where $b$ is the braid index of $L$.  We call the number $(d_+-d_-)/2 + 1$ the {\em MFW bound}.

\begin{theorem}\label{thm:MFWbounds}
Every manifold in $\calT_2$ has a positive braid representative that realizes the MFW bound.\\
The braidwords listed in Tables~\ref{tab:braidwords}, \ref{tab:braidwordsT2}, and \ref{tab:braidwordsT2-2} all realize the MFW bound.
\end{theorem}

\bigskip
We conclude the introduction with a few questions.

\smallskip

Observe that the closure of the positive braid shown in Figure~\ref{fig:braid} is an asymmetric hyperbolic L-space knot of genus 12 with braid index 4 and tunnel number 2. This knot (the manifold \mfld{t12533}) is the only knot in $\calA$ with braid index $4$.  In $\calT_2-\calA$, three more knots (the manifolds \mfld{t09284}, \mfld{t10496}, \mfld{o9\_34409}) also have braid index $4$.  All others in $\calT_2$ have larger braid indices.

Knots with braid index $2$ are torus knots, and the positive ones are L-space knots. The L-space knots with braid index $3$ are either positive torus knots or  braid positive twisted torus knots \cite{Lee}. In addition to being braid positive, such knots are all strongly invertible and have tunnel number 1 (e.g.\ via \cite{Dean2003}).  Hence asymmetric L-space knots and, more generally, L-space knots with tunnel number greater than $1$ must have braid index at least $4$.

\begin{question}
Are there other asymmetric L-space knots of braid index 4?
\end{question}

\begin{question}
What are the L-space knots of braid index 4 with tunnel number greater than $1$?
\end{question}

The smallest genus of asymmetric L-space knot in $\calA$ is $12$ which is also realized by only the knot \mfld{t12533}.

\begin{question}
What is the smallest genus among asymmetric L-space knots?
\end{question}

The previously known examples of asymmetric L-space knots all admit an {\em alternating surgery}, a non-trivial surgery to the double branched cover of a non-split alternating link \cite{Baker2017}.

\begin{question}
Do any of the asymmetric L-space knots in $\calA$ admit an alternating surgery? \footnote{Recent work shows that none of the knots in $\calA$ admit an alternating surgery \cite{BKM}.}
\end{question}

\section{Methods}
For most calculations, we use the `kitchen sink' prepackaged Docker image \cite{kitchensink}\footnote{See also \url{https://snappy.math.uic.edu/installing.html#kitchen-sink}}  for running SnapPy \cite{snappy} within Sage \cite{sagemath} and Python \cite{python} alongside Berge's Heegaard program \cite{heegaard}. Throughout we assume SnapPy has been initialized with \mintinline{python}{import snappy}. When finding initial diagrams and braid presentations, we also employ the standalone SnapPy application \cite{snappy} and Frank Swenton's KLO \cite{KLO}.

\subsection{Census L-space knots}
First we extract the collection $\calD$ of Dunfield's L-space knots.  Using Dunfield's data \texttt{exceptional\_fillings.csv} from \cite{Dunfield2018} and \texttt{QHSolidTori.csv} from \cite{Dunfield2019}, one may obtain lists of these $630$ L-space knots and $2$ unclassified knots that comprise our collection $\calD$ using Pandas \cite{pandas} for database queries as follows:

\begin{minted}[mathescape,
	%linenos,
	numbersep=5pt,
	gobble=1,
	frame=lines,
	framesep=2mm]{python}
	import pandas
	exfil = pandas.read_csv("exceptional_fillings.csv")
	qhst = pandas.read_csv("QHSolidTori.csv")
	S3knot = exfil.loc[(exfil['kind'] == 'S3')]['cusped'].to_list();
	knownFloerSimple = qhst.loc[(qhst['floer_simple'] == 1)]['name'].to_list();
	unknownFloerSimple = qhst.loc[(qhst['floer_simple'] == 0)]['name'].to_list();
	Lspaceknot = list(set(S3knot) & set(knownFloerSimple));
	Lspaceknot.sort();  Lspaceknot.sort(key=len); # for ordering
	maybeLspaceknot = list(set(S3knot) & set(unknownFloerSimple));	
	D = Lspaceknot + maybeLspaceknot
\end{minted}

This gives the list \verb|Lspaceknot| of 630 known L-space knot complements in the SnapPy census of hyperbolic manifolds assembled from at most $9$ ideal tetrahedra. %, presented at the end of this article in \S~\ref{sec:theknots}.
The two knot complements which have not yet been confirmed to be complements of L-space knots are given the list \verb|maybeLspaceknot| which is $\{ \mfld{o9\_30150}, \mfld{o9\_31440} \}$.  The two lists are concatenated as \verb|D| (which we also write as $\calD$).   The SnapPy census names for the manifolds of $\calD$ are listed in Tables~\ref{tab:censusLspaceknotsI} and \ref{tab:censusLspaceknotsII}.

\subsection{Symmetries}
To obtain the collection $\calA$ of asymmetric L-space knots we check the orders of the symmetry groups of the manifolds in $\calD$.  
In particular, the code
\begin{minted}[breaklines]{python}
	A=[mfld for mfld in D if snappy.Manifold(mfld).symmetry_group().order() == 1]
\end{minted}
returns the following set of  9 manifolds:
\begin{align*}
	\calA = \{& \mfld{t12533}, \mfld{t12681}, \mfld{o9\_38928}, \mfld{o9\_39162}, \mfld{o9\_40363},\\
	& \phantom{xxxxxxxxxxxxxxxxx} \mfld{o9\_40487}, \mfld{o9\_40504}, \mfld{o9\_40582}, \mfld{o9\_42675} \}
\end{align*}

\begin{lemma}\label{lem:symmetry}
The symmetry groups of the manifolds of $\calD-\calA$ are $\Z/2$, generated by a strong involution.
\end{lemma}

\begin{proof}
First, the command
\begin{minted}{python}
	[mfld for mfld in D 
	        if snappy.Manifold(mfld).symmetry_group().order() > 2]
\end{minted}
returns an empty list, showing that the manifolds in $\calD$ are either asymmetric or have symmetry group $\Z/2$. Next, the command \verb|.is_invertible_knot()| checks whether a one-cusped manifold is strongly invertible. As the code
\begin{minted}[mathescape,
	linenos,
	numbersep=5pt,
	gobble=1,
	frame=lines,
	framesep=2mm]{python}
	noninvertible=[]
	for mfld in D:
	    S=snappy.Manifold(mfld).symmetry_group()
	    if S.is_invertible_knot() == False:
	        noninvertible.append(mfld)
	noninvertible
\end{minted}
\noindent
returns just the list of the 9 manifolds of $\calA$,  each of these $Z/2$ symmetry groups of the manifolds in $\calD-\calA$ is generated by a strong involution as claimed.
\end{proof}

\subsection{Tunnel numbers}

\begin{proof}[Proof of Theorem~\ref{thm:tunnelnumber}]
This proof splits into two parts.  In Part 1 we use SnapPy and Berge's Heegaard to (a) show every manifold of $\calD$ has tunnel number at most $2$ and (b) identify a subset $\calT_2 \subset \calD$ where every manifold of its complement $\calD-\calT_2$ has tunnel number $1$.  Necessarily $\calA \subset \calT_2$ since the asymmetric manifolds cannot have tunnel number $1$.  In Part 2 we show the manifolds of $\calT_2 - \calA$ actually have tunnel number $2$ by observing they have a toroidal Dehn filling that fails Kobayashi's criteria for having Heegaard genus $2$.  

\subsubsection{Part 1: Bounding tunnel number}
Here we use SnapPy and Berge's Heegaard to determine upper bounds on the tunnel numbers of manifolds in $\calD$.
The tunnel number of a link complement is one less than the Heegaard genus of the link exterior's splittings into a handlebody and a compression body. SnapPy provides the presentation of a manifold's fundamental group, and Heegaard checks whether that presentation is realized by a Heegaard splitting consisting of a handlebody and a compression body.   Since only the unknot has tunnel number $0$, once we find a genus $2$ Heegaard splitting for a knot complement, we know the knot has tunnel number $1$.  Further obstructions (such as the absence of certain symmetries) are required to confirm a knot for which only genus $3$ Heegaard splittings were found actually has tunnel number $2$ and not $1$.  

First we select the manifolds in $\calD$ whose fundamental group SnapPy presents with $2$ generators and Berge's Heegaard  confirms is actually realized by a genus $2$ Heegaard splitting.

\begin{minted}[mathescape,
	linenos,
	numbersep=5pt,
	gobble=0,
	frame=lines,
	framesep=2mm]{python}
import heegaard
TN1easy=[];
TNmaybemore=[];
for mfld in D:
    M=snappy.Manifold(mfld)
    G=M.fundamental_group()
    if G.num_generators()==2 and heegaard.is_realizable(G.relators()):
        TN1easy.append(mfld)
    else:
        TNmaybemore.append(mfld)
\end{minted}

This produces a list \verb|TN1easy| consisting of $473$ of the $632$ manifolds in $\calD$ which are easily confirmed to have tunnel number $1$.  The remaining  $159$ which may have tunnel number greater than $1$ (including those $9$ of $\calA$ which we already know have tunnel number $2$) are collected in the list \verb|TNmaybemore|.
        
Note that the tunnel number of a knot is bounded above by the tunnel number of any link of which it is a component.  So next we examine the tunnel numbers of the manifolds obtained by drilling a dual curve from the manifolds in \verb|TNmaybemore|.  (The {\em dual curves} are curves in the $1$--skeleton dual to the underlying triangulation of a manifold in SnapPy.  Such curves are frequently simple geodesics and may be drilled from the manifold with SnapPy.) 

\begin{minted}[mathescape,
	linenos,
	numbersep=5pt,
	gobble=0,
	frame=lines,
	framesep=2mm]{python}
TN1drilleasy=[];
TN2probably=[];
for mfld in TNmaybemore:
    M=snappy.Manifold(mfld)
    dclength=len(M.dual_curves())
    for i in range(dclength):
        N=M.drill(i)
        G=N.fundamental_group()
        if G.num_generators()==2 and heegaard.is_realizable(G.relators()):
            TN1drilleasy.append([mfld,i])
            break
    if TN1drilleasy[-1][0] != mfld:
       TN2probably.append(mfld)
\end{minted}

This produces a list \verb|TN1drilleasy| which records a manifold of \verb|TNmaybemore| and the index of its first dual curve for which SnapPy and Heegaard confirmed a genus $2$ Heegaard splitting of the drilling of that dual curve. In particular, the $137$ manifolds of \verb|TN1drilleasy| all have tunnel number 1.  The remaining $22$ manifolds are collected in \verb|TN2probably| which we also denote as $\calT_2$.  Note that $\calT_2$ necessarily contains $\calA$.

Now we confirm that the $22$ manifolds of $\calT_2 =$\verb|TN2probably| actually have genus $3$ Heegaard splittings and thus have tunnel number at most $2$.

\begin{minted}[mathescape,
	linenos,
	numbersep=5pt,
	gobble=0,
	frame=lines,
	framesep=2mm]{python}
TNcheck=[]
for mfld in TN2probably:
    G=snappy.Manifold(mfld).fundamental_group()
    TNcheck.append([mfld, G.num_generators(), heegaard.is_realizable(G.relators())])
\end{minted}

This produces a list \verb|TNcheck| which shows that each manifold in $\calT_2 =$ \verb|TN2probably| has a presentation of its fundamental group with $3$ generators that is realized by a Heegaard splitting.  Therefore these manifolds have tunnel number at most $2$ as claimed.

Finally, if any manifold in $\calT_2$ had tunnel number $1$, then it would have a genus $2$ Heegaad splitting.  In particular, the manifold would then admit a strong involution induced by the hyperelliptic involution of the genus $2$ Heegaard surface.  Therefore, since the manifolds of $\calA \subset \calT_2$ are asymmetric, they cannot have tunnel number $1$ and must have tunnel number $2$.

%  A=
%  ['t12533',
%  't12681',
%  'o9_38928',
%  'o9_39162',
%  'o9_40363',
%  'o9_40487',
%  'o9_40504',
%  'o9_40582',
%  'o9_42675']
 
%  TN2probably - A = 
%  ['t09284',
%  't09450',
%  't09633',
%  't10496',
%  'o9_28751',
%  'o9_29751',
%  'o9_32314',
%  'o9_33380',
%  'o9_33944',
%  'o9_33959',
%  'o9_34409',
%  'o9_36380',
%  'o9_40026']

\subsubsection{Part 2: Confirming tunnel numbers}

As noted in Part 1, the tunnel number of a knot is one less than the Heegaard genus of the knot exterior.  Since Heegaard genus may only decrease upon Dehn filling, the Heegaard genus of any Dehn filling gives a lower bound on the tunnel number of a knot plus one.  Knowing  that the manifolds of $\calT_2 - \calA$ have tunnel number at most $2$, we will confirm they have tunnel number $2$ by demonstrating they have a Dehn filling of Heegaard genus at least $3$.

\medskip
It so happens that each manifold of $\calT_2 - \calA$ admits at least one Dehn filling to a graph manifold, a closed orientable $3$--manifold with essential tori that decompose it into Seifert fibered spaces. These fillings may be observed through Dunfield's survey of exceptional surgeries \cite{Dunfield2018}.  For each manifold of $\calT_2 - \calA$, a Dehn filling producing a graph manifold and the Regina notation \cite{regina} for that graph manifold is listed in Table~\ref{tab:graphmanifolds}. Four of the manifolds of $\calT_2 - \calA$ have another graph manifold filling which is not listed as those filled manifolds have $2$ generator fundamental group presentations that correspond to a genus $2$ Heegaard splitting, as confirmed by Heegaard.  Each filling in Table~\ref{tab:graphmanifolds}, except for the filling of {\tt o9\_29751} marked with $\dagger$, is a graph manifold where two Seifert fibered spaces are glued together along a single torus.  We will address these now and the $\dagger$ marked one thereafter.

\bigskip
Kobayashi classifies the closed orientable $3$--manifolds with Heegaard genus $2$ that contain an essential torus \cite{kobayashi}.  
Genus $2$ manifolds with a single separating torus in its torus decomposition are formed from either
\begin{itemize}
    \item[(i)] a Seifert fibered space over the disk with 2 exceptional fibers and the exterior of a $1$--bridge knot in a lens space,
    \item[(ii)] a Seifert fibered space over the \mobius band with at most 2 exceptional fibers and the exterior of a $2$--bridge knot in $S^3$, or 
    \item[(iii)] a Seifert fibered space over the disk with 2 or 3 exceptional fibers and the exterior of a $2$--bridge knot in $S^3$. 
\end{itemize}
In each case, in the common torus the regular fiber is identified with the meridian of the knot. (The remaining two cases of the main theorem of \cite{kobayashi} involve manifolds with a torus decomposition involving either (iv) two tori or (v) a non-separating torus.)
Since we are considering graph manifolds, these exteriors of $2$--bridge knots and $1$--bridge knots in lens spaces must be Seifert fibered.  Since a meridional filling of these knot exteriors must extend to a Seifert fibration of $S^3$ or a lens space, their meridians must intersect their regular fibers just once.  Thus, if a graph manifold obtained from gluing two Seifert fibered spaces together along an incompressible torus has Heegaard genus $2$, then the regular fibers of the two pieces in the common torus intersect just once.

The graph manifolds that appear in Table~\ref{tab:graphmanifolds} (not marked by $\dagger$) have Regina descriptions of the type 
\begin{center}
    {\tt SFS [D: (p,q) (r,s)] U/m SFS [D: (t,u) (v,w)], m = [ a,b | c,d ]} 
\end{center}
or
\begin{center}
     {\tt SFS [D: (p,q) (r,s)] U/m SFS [M/n2: (t,u)], m = [ a,b | c,d ]}.
\end{center}
Regina uses the normalization for Seifert fibered manifolds with boundary where each exceptional fiber is described with a relatively prime pair $(x,y)$ so that $0<y<x$.  On the boundary torus, a basis is given in terms of a fiber $(1,0)$ and a curve $(0,1)$ representing the base orbifold. With two Seifert fibered spaces glued together along a common boundary torus, the matrix $m$ expresses the basis from the second Seifert fibered space in terms of the basis from the first. In particular, the regular fiber of the second is the curve $(1,0)m = (a,b)$ (the first row of $m$) with respect to the basis from the first.
Consequently, the regular fibers of the two pieces of these kinds of graph manifolds intersect just once exactly when $b$ (the upper right entry of $m$) is $\pm1$.  

A quick check shows that among the manifolds in Table~\ref{tab:graphmanifolds}, only the two of type 
\begin{center}
    {\tt SFS [D: (p,q) (r,s)] U/m SFS [M/n2: (t,u)], m = [ a,b | c,d ]}
\end{center} 
have $b=\pm1$ in their gluing matrix $m$.  However, since the second piece is a Seifert fibered space over the \mobius band, its regular fiber must be identified with a slope that gives an $S^3$ filling of the first piece.  Both of the first pieces are exteriors of trefoils and have meridians of slope $(1,1)$.  Hence $(1,1)m = (a+c,b+d)$ must be the fiber slope $(1,0)$ for the graph manifold to have Heegaard genus $2$.  Since $b=\pm1$ and $d=0$ in these two cases, neither have genus $2$.  

The Seifert fibered pieces of the form {\tt SFS [D: (2,1) (2,1)]} admit an alternative fibration as {\tt SFS [M/n2:]}. Aside from the solid torus, all other Seifert fibered spaces with boundary have unique Seifert fibrations.  In Regina's parameterization of boundary curves of {\tt SFS [D: (2,1) (2,1)]}, the regular fiber of {\tt SFS [M/n2:]} has slope $(1,1)$.  Thus, when using the alternative fibration on the first piece of
\begin{center}
 {\tt SFS [D: (2,1) (2,1)] U/m SFS [D: (t,u) (v,w)], m = [ a,b | c,d ]},
\end{center}
the regular fibers of the two pieces of these kinds of graph manifolds intersect just once exactly when $b+d$  is $\pm1$.  Five of the graph manifolds in Table~\ref{tab:graphmanifolds} have a {\tt SFS [D: (2,1) (2,1)]} piece, but only three of them have $b+d = \pm1$.  Of those three, the other Seifert fibered piece is not the exterior of a $(2,n)$--torus knot. 

Taken together, none of the manifolds in Table~\ref{tab:graphmanifolds}  (aside from  {\tt o9\_29751(1,1)}) admit presentations satisfying any of the three criteria of Kobayashi listed above.  Hence these manifolds must have Heegaard genus at least $3$. Thus the manifolds of $\calT_2-\calA$ (except \mfld{o9\_29751}) must have tunnel number at least $2$.  Hence their tunnel numbers are all exactly $2$.

\bigskip
The $\dagger$ marked filling {\tt o9\_29751(1,1)} in Table~\ref{tab:graphmanifolds} has a torus decomposition that almost fits Kobayashi's criterion (iv) for a genus $2$ manifold, but it fails the gluing requirements. Unfortunately, this failure's dependence upon understanding Regina's orientation conventions is somewhat delicate. So rather than discussing that we observe {\tt o9\_29751(1,1)} cannot have Heegaard genus $2$ by other means.

The knot exterior \mfld{o9\_29751} has a surgery description as {\tt L10n72(3,-4)(0,0)(3,-2)} where the filling {\tt o9\_29751(p,q)} corresponds to {\tt L10n72(3,-4)(-p+q,q)(3,-2)}. (See Table~\ref{tab:surgerydescriptions} and the discussion in Section~\ref{sec:diagrams}.)  Since {\tt L10n72} is a strongly invertible link, any Dehn filling of \mfld{o9\_29751} can be viewed as the double branched cover of the corresponding rational tangle filling of the branch locus of the two-fold quotient of {\tt L10n72(3,-4)(0,0)(3,-2)}.   

Figure~\ref{fig:L12n243} shows the $2$-component link {\tt L12n243} whose double branched cover is the manifold {\tt o9\_29751(1,1)}. The green arc is the core arc of the rational tangle filling; its exterior is a tangle whose double branched cover is \mfld{o9\_29751}.  (One may further observe that the torus decomposition of {\tt o9\_29751(1,1)} is reflected in the Conway sphere decomposition of {\tt L12n243} also indicated in Figure~\ref{fig:L12n243}.)
Replacing a neighborhood of the green arc by another rational tangle corresponds to Dehn surgery along the knot that is the lift of the arc in the double branched cover, and hence to a Dehn filling of {\tt L12n243}. 
Observe now that any link resulting from a rational tangle replacement has one component that is an unknot and another that is the knot {\tt 8\_21}.    
Since the knot {\tt 8\_21} is a Montesinos knot (and not a two-bridge knot), it has bridge number at least $3$.  Therefore {\tt L12n243} and any two-component link obtained by such a rational tangle replacement must have bridge number at least $4$. Indeed, one can find bridge number $4$ presentations for these links.

Since {\tt o9\_29751} has symmetry group $\Z/2$ as we determined in Lemma~\ref{lem:symmetry}, Thurston's Hyperbolic Dehn Filling Theorem implies that all but finitely many Dehn fillings of {\tt o9\_29751} will also have symmetry group $\Z/2$.  (See \cite[Theorem 5.2]{Auckly} for example.) Therefore, for all but finitely many of the two-component links obtained by rational tangle replacements along the green arc in Figure~\ref{fig:L12n243}, there is no other knot or link that has the same double branched cover. Yet if {\tt o9\_29751} were to have tunnel number $1$, then any Dehn filling would have Heegaard genus at most $2$ and therefore be the double branched cover of a knot or link with a $3$--bridge presentation.  This cannot be since all but finitely many of these links have bridge number at least $4$.   Therefore the tunnel number of {\tt o9\_29751} must be at least $2$.  Hence {\tt o9\_29751} must have tunnel number exactly $2$.

 \begin{figure}
     \centering
     \includegraphics[height = 6cm]{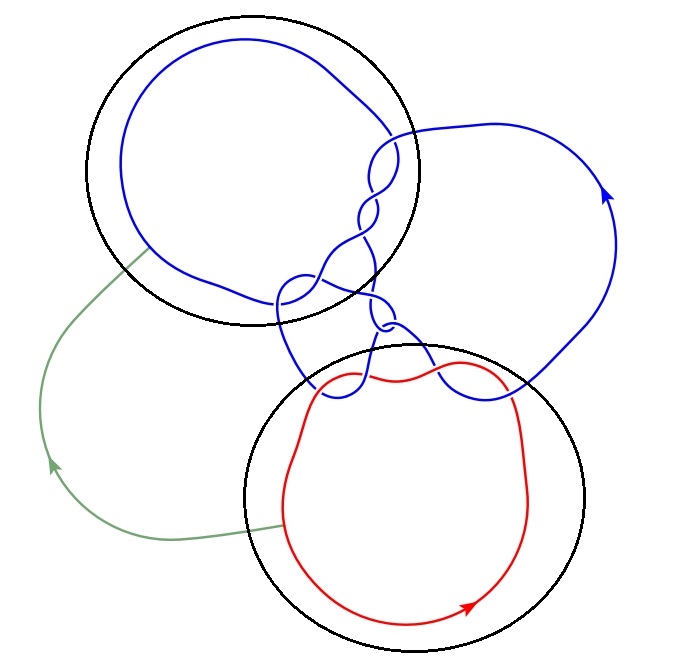}
     \caption{The link {\tt L12n243} is shown with its decomposition along Conway spheres.   The double branched cover of this link is the graph manifold {\tt o9\_29751(1,1)}.  The green arc lifts to the core curve of the Dehn filling.}
     \label{fig:L12n243}
 \end{figure}

This completes the proof of Theorem~\ref{thm:tunnelnumber}.
\end{proof}

\begin{table}[]
    \centering
    \caption{For each manifold of $\calT_2-\calA$, a Dehn filling that yields a graph manifold is listed.  The graph manifold is given in the notation of Regina.  }
    \label{tab:graphmanifolds}

    \begin{tabular}{>{\ttfamily}l >{\ttfamily}l}
    \toprule
    {\rm Dehn filling} & {\rm Graph Manifold}\\
    \midrule
t09284(0, 1)	&	SFS [D: (2,1) (2,1)] U/m SFS [D: (2,1) (3,2)], m = [ -1,3 | -1,2 ]	\\

t09450(1, 1)	&	SFS [D: (2,1) (2,1)] U/m SFS [D: (3,1) (3,2)], m = [ 1,-2 | 0,1 ]	\\

t09633(0, 1)	&	SFS [D: (2,1) (3,1)] U/m SFS [D: (2,1) (3,1)], m = [ 1,2 | 0,1 ]	\\

t10496(0, 1)	&	SFS [D: (2,1) (2,1)] U/m SFS [D: (3,1) (3,2)], m = [ -1,2 | -1,1 ]	\\

o9\_28751(1, 1)	&	SFS [D: (2,1) (2,1)] U/m SFS [D: (3,2) (5,2)], m = [ 1,-2 | 0,1 ]	\\

$\dagger$ o9\_29751(1, 1)	&	SFS [D: (2,1) (2,1)] U/m SFS [A: (2,1)] U/n SFS [D: (2,1) (3,2)],\\ & m = [ 0,-1 | 1,0 ], n = [ 1,1 | 0,1 ]	\\

o9\_32314(0, 1)	&	SFS [D: (2,1) (3,1)] U/m SFS [M/n2: (4,3)], m = [ 0,1 | 1,0 ]	\\

o9\_33380(0, 1)	&	SFS [D: (2,1) (3,2)] U/m SFS [D: (2,1) (3,2)], m = [ -2,3 | -1,1 ]	\\

o9\_33944(0, 1)	&	SFS [D: (2,1) (2,1)] U/m SFS [D: (2,1) (3,2)], m = [ 2,3 | 1,2 ]	\\

o9\_33959(0, 1)	&	SFS [D: (2,1) (3,2)] U/m SFS [D: (2,1) (3,2)], m = [ 1,2 | 0,1 ]	\\

o9\_34409(0, 1)	&	SFS [D: (2,1) (3,1)] U/m SFS [D: (3,1) (3,2)], m = [ -1,2 | -1,1 ]	\\

o9\_36380(0, 1)	&	SFS [D: (2,1) (3,2)] U/m SFS [M/n2: (5,3)], m = [ 0,-1 | 1,0 ]	\\

o9\_40026(0, 1)	&	SFS [D: (2,1) (3,1)] U/m SFS [D: (2,1) (5,3)], m = [ -1,2 | -1,1 ]	\\
\bottomrule
    \end{tabular}
\end{table}

\subsection{Diagrams and positive braids}\label{sec:diagrams}

Presently, SnapPy only has diagrams for knots up to 15 crossings and links up to 14 crossings, tabulated by \cite{Hoste1998}. However, hyperbolic knot complements assembled from up to $8$ ideal tetrahedra are catalogued by \cite{CDW, CKP, CKM} and presented with diagrams or as members of certain families.  Dunfield's work \cite{Dunfield2018} extends this catalogue to hyperbolic knot complements assembled from $9$ ideal tetrahedra, but it does not offer any diagram or familiar presentation of them.

Of the knots in $\calD$, only $263$ admit ideal triangulations with at most $8$ ideal tetrahedra.  Diagrams for these may be obtained from  \cite{CDW, CKP, CKM}.  In general, diagrams for the remaining $369$ knots in $\calD$ have not been determined.

Among the $22$ knots in $\calT_2$, only six may be assembled from fewer than $9$ ideal tetrahedra, and hence diagrams have already been determined for them. In $\calA$ are the two manifolds $\mfld{t12533}$ and $\mfld{t12681}$ which are listed as the knots $\mfld{K8\_290}$ and $\mfld{K8\_296}$ in the Champanerkar-Kofman-Mullen tabulation \cite{CKM}.  There, these two knots are presented as the generalized twisted torus knots $T(5, 6, 4, -2, 3, 5)$ and $T(7, 3, 5, 5, 4, 4)$. 
In $\calT_2-\calA$ are the four manifolds $\mfld{t09284}$, $\mfld{t09450}$, $\mfld{t09633}$, and $\mfld{t10496}$ which are listed as the knots $\mfld{K8\_186}$, $\mfld{K8\_189}$, $\mfld{K8\_195}$, and $\mfld{K8\_220}$.  These are presented as the generalized twisted torus knots $T(5, 6, 3, -1, 2, 2)$, $T(10, 6, 4, 3)$, $T(7, 5, 5, 3)$, and $T(10, 3, 4, 3, 2, -3)$.  See \cite{CKM} for the notation.

\begin{figure}
\begin{tikzpicture}
\pic[rotate=90,
braid/.cd,
every strand/.style={ultra thick},
height=.6cm,
width=.43cm,
gap=0.2] {braid={s_1 s_1 s_2 s_2 s_1 s_2 s_2 s_2 s_2 s_2 s_2 s_2 s_2 s_2 s_1 s_2 s_2 s_3 s_2 s_1 s_1 s_2 s_2 s_1 s_3 s_2 s_2}};
\end{tikzpicture}
\caption{A positive braid whose closure is the hyperbolic asymmetric L-space knot with complement 't12533'.}
\label{fig:braid}
\end{figure}
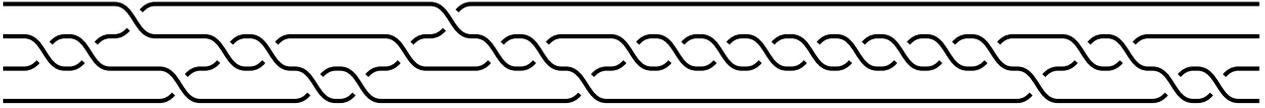

\medskip

The original goal of this note was  to  find not only diagrams for the knots of $\calA$ but  also presentations of these knots as closures of  positive braids.  Theorem~\ref{thm:asymbraidpos} records the achievement of this goal  and Theorem~\ref{thm:TN2braidpos} records its extension to all of $\calT_2$ which we prove below.
Note that the identification of  the manifolds \mfld{t12681}, \mfld{t09450}, and \mfld{t09633} as the complements of the generalized twisted torus knots $T(7, 3, 5, 5, 4, 4)$, $T(10, 6, 4, 3)$, and $T(7, 5, 5, 3)$ 
by \cite{CKM} already achieve this goal for these knots.

\begin{proof}[Proof of Theorem~\ref{thm:asymbraidpos} and Theorem~\ref{thm:TN2braidpos}]
For each manifold of $\calA$, 
Table~\ref{tab:braidwords} gives a positive braid word whose closure is a knot with the manifold as its complement.  For the remaining manifolds of $\calT_2-\calA$, Tables~\ref{tab:braidwordsT2} and \ref{tab:braidwordsT2-2} give positive braid words.   As an example, Figure~\ref{fig:braid} illustrates a positive braid whose closure has
 \mfld{t12533} 
as its complement.  The braid word is given as a list of positive integers where the integer $n$ represents the standard braid generator $\sigma_n$.  These may be verified using SnapPy by the command 
\begin{minted}{python}
	snappy.Link(braid_closure=WORD).exterior().identify()
\end{minted}
where \mintinline{python}{WORD} 
is the list of integers giving the braid word being checked.
 The output will be the list of names that SnapPy has for the manifold that is the complement of the closure of the braid.
\end{proof}

The above proof only gives confirmation that positive braidwords for the manifolds of $\calT_2$ have indeed been found.
To actually obtain the braid words, we need to first find a diagram for the knot.

Given a triangulated $1$--cusped manifold with a known $S^3$ filling, it must be the complement of a knot in $S^3$.  There are currently no implemented algorithms for computing a knot diagram given this data\footnote{This problem will appear on a problem list compiled from the ICERM workshop Perspectives on Dehn Surgery.}, so we attempt to find a diagram for the knot by the following method:

Using SnapPy, we drill out a small number of short geodesics (dual curves) until we obtain a manifold that SnapPy recognizes as the complement of a hyperbolic link for which it has a corresponding link diagram. These are prime links of at most $14$ crossings.  SnapPy can confirm there are isometries between the drilled manifold and the link complement while supplying the actions of these isometries on the cusps.  For each isometry, the action on the cusps determines a slope for each link component that is either the image of a meridian of a drilled-out geodesic or the original $S^3$ filling slope.

Moving to KLO, we create a surgery diagram from the SnapPy link diagram with surgery coefficients given by the slopes.  We use KLO to assist in the reduction of the surgery diagram into a diagram of the knot (with slope $\infty$). In principle, this can be done through a sequence of Rolfsen twists, adding and removing $\infty$--sloped unknotted components as needed \cite{rolfsen}.

For the manifolds of $\calT_2$, we only needed to drill at most two geodesics before SnapPy identified the resulting manifold as the complement of a link of unknots.  %This is shown in Table \ref{tab:drilled}. 
Furthermore, for each manifold except \mfld{o9\_28751} we were able to find
\begin{itemize}
\item an isometry of the drilled manifold to the link complement that took the $S^3$ filling slope of the original cusp to an $\infty$ slope of one of the link complements, and
\item restricting this link to the sublink of the remaining components yielded either an unknot or a Hopf link.
\end{itemize}
Such surgery descriptions of all the manifolds in $\calT_2$ are given in Table~\ref{tab:surgerydescriptions}.
For two drillings of \mfld{o9\_28751}, the only recognized link is  \mfld{L12n2002} where every sublink of two-components is a $(2,4)$ torus link. Up to symmetries, this gives the surgery description $L12n2002(1,0)(-5,2)(-3,2)$ of \mfld{o9\_28751} which may be reduced after an appropriate insertion of an unknot.  However, examining three drillings of \mfld{o9\_28751} yields $L14n60453(-2,3)(3,-1)(1,0)(-1,2)$ where the surgery sublink is just a chain link that may be reduced by a sequence of Rolfsen twists.

\begin{table}[h]
\caption{Surgery descriptions for the knots in $\calA$ (left) and $\calT_2-\calA$ (right).  The knot corresponds to the component with the $(1,0)$ filling}
\label{tab:surgerydescriptions}
    \centering
\begin{tabular}{ccc}  
    \begin{tabular}{ll}
    \toprule
    Manifold & Surgery Description\\
    \midrule
    \mfld{t12533} & L14n58444(5,2)(1,2)(1,0)\\
    \mfld{t12681} & L12n1968(1,3)(7,2)(1,0)\\
    \mfld{o9\_38928} & L13n9833(-3,1)(1,-2)(1,0)\\
    \mfld{o9\_39162} & L12n1968(1,4)(5,1)(1,0)\\
    \mfld{o9\_40363} & L12n1968(1,0)(7,2)(1,4)\\
    \mfld{o9\_40487} & L13n8037(1,-2)(-1,1)(1,0)\\
    \mfld{o9\_40504} & L13n9833(-2,1)(1,-3)(1,0)\\
    \mfld{o9\_40582} & L14n58444(-1,1)(2,-1)(1,0)\\
    \mfld{o9\_42675} & L11n425(-1,-2)(3,1)(1,0)\\
    \bottomrule\\
    \\
    \\
    \\
    \end{tabular}

&
\quad 
&

    \begin{tabular}{ll}
    \toprule
    Manifold & Surgery Description\\
    \midrule   
    \mfld{t09284} & L10n72(1,-4)(1,0)(-3,1)\\
    \mfld{t09450} & L13n9547(1,-2)(-3,2)(1,0)\\
    \mfld{t09633} & L11n345(1,-4)(5,-1)(1,0)\\
    \mfld{t10496} & L12n1925(3,-2)(1,-1)(1,0)\\
    \mfld{o9\_28751} & L14n60453(-2,3)(3,-1)(1,0)(-1,2) \\
    \mfld{o9\_29751} & L10n72(3,-4)(1,0)(3,-2)\\
    \mfld{o9\_32314} & L10n86(1,0)(1,4)(5,1)\\
    \mfld{o9\_33380} & L13n7625(1,2)(3,2)(1,0)\\
    \mfld{o9\_33944} & L13n7625(1,2)(7,3)(1,0)\\
    \mfld{o9\_33959} & L14n56927(3,1)(1,2)(1,0)\\
    \mfld{o9\_34409} & L12n1952(5,-4)(-1,1)(1,0)\\
    \mfld{o9\_36380} & L10n86(1,0)(2,5)(7,3)\\
    \mfld{o9\_40026} & L11n347(1,4)(-1,0)(5,1)\\
    \bottomrule
    \end{tabular}

     \end{tabular}
\end{table}

Since the surgery coefficients on this unknot or Hopf sublink must present $S^3$, reduction of this sublink to the empty link by Rolfsen twists is straightforward. When this sublink is an unknot, its surgery coefficient must be of the form $-1/n$ for some integer $n$, so it may be eliminated by performing $n$ Rolfsen twists upon it.  When this sublink is a Hopf link, then one may reduce surgery coefficients by performing a sequence of Rolfsen twists alternately on each component until one component has a surgery coefficient of the form $-1/n$, which then may be eliminated by performing $n$ Rolfsen twists.  Thereafter the remaining unknot component may be eliminated.  For the surgery descriptions given in Table~\ref{tab:surgerydescriptions}, at least one component of the surgery sublink already has a surgery coefficient of the form $-1/n$ except for \mfld{o9\_29751} and \mfld{o9\_36380}.

It turned out that in many cases the last remaining component of the sublink was either a positive braid axis for the resulting knot or quite close to being one.  While in each case a positive braid presentation can be found by hand without too much trouble, KLO and SnapPy can aid in finding such a presentation.  The knot can be transferred from KLO to SnapPy by exporting and importing a PLink file, where then SnapPy can use the diagram to find a braid word whose closure is the link.  After possibly having SnapPy simplify the diagram, in each case SnapPy coincidentally found a positive (or negative) braid word. %either immediately or with minimal adjustment in KLO. 

While the search for surgery diagrams can be scripted in SnapPy, the manipulation in KLO presently requires human intervention.

\subsection{The MFW bound on braid index}

\begin{proof}[Proof of Theorem~\ref{thm:MFWbounds}]
Recall from the introduction that the MFW bound on the braid index of a link 
is $(d_+ + d_-)/2 + 1$ where $d_+$ and $d_-$ are the max and min degrees of $v$ in the $P(v,z)$ version of the HOMFLY-PT polynomial.
Sage has routines in its Link class to compute a $P(a,z)$ version of the polynomial which is related to the $P(v,z)$ version by $a=v^{-1}$.  Hence the average of the max and min degrees of $v$ in the MFW bound equals average of the max and min degrees of $a$ in the polynomial $P(a,z)$.  Furthermore, rather than asking for the minimum degree of $a$ in $P(a,z)$, we may instead ask for the maximum degree of $a$ in $P(a^{-1},z)$.

With our list of braid words for the manifolds of $\calT_2$, we may then enter our links into Sage, compute the MFW bounds, and compare to the braid index of our braid words.
\begin{minted}[mathescape,
	linenos,
	numbersep=5pt,
	gobble=0,
	frame=lines,
	framesep=2mm]{python}
output = []

for word in wordlist:
    L = Link(B(word))
    hp = L.homfly_polynomial('a','z','az')
    braidindex = max(word)+1
    MFWbound = hp(a,z).degree(a) + hp(a^-1,z).degree(a))/2 + 1
    data=[snappy.Link(L).exterior().identify(),  (braidindex, MFWbound)]
    output.append(data)

print(output)
\end{minted}

For the words which do not realize the MFW bounds, we allow SnapPy to work harder at attempting to simplify the diagram.  

\begin{minted}[mathescape,
	linenos,
	numbersep=5pt,
	gobble=0,
	frame=lines,
	framesep=2mm]{python}
L=snappy.Link(braid_closure=word)  # word is a braid word that needs reduction
L.simplify(mode='global',type_III_limit=1000)
newword=L.braid_word()
braidindex=max(newword)+1
print(braidindex, newword)
\end{minted}

In all cases, we find that this succeeds in producing positive braids that realize the MFW bound.
\end{proof}

\subsection{An illustration of the passage from manifold to diagram}
We illustrate this process with the manifold \mfld{o9\_40504}.
The code
\begin{minted}[mathescape,
	linenos,
	numbersep=5pt,
	gobble=0,
	frame=lines,
	framesep=2mm]{python}
M = snappy.Manifold('o9_40504')
for i in range(3):
    print(M.drill(i).identify())
    
for i in range(3):
    print(M.drill(0).drill(i).identify())
\end{minted}
informs us that two simple drillings are isometric to known link complements:  The drilling \verb|M-0-1| is isometeric to \mfld{L13n9833} while the drilling \verb|M-0-2| is isometric to \mfld{L11n425}.
We shall work with the first of these.  Writing
\begin{minted}[mathescape,
	linenos,
	numbersep=5pt,
	gobble=0,
	frame=lines,
	framesep=2mm]{python}
M.drill(0).drill(1).is_isometric_to(snappy.Manifold('L13n9833'),1)
\end{minted}
lists the action on the cusps of the isometries between \verb|M-0-1| and \mfld{L13n9833} as
\begin{minted}[mathescape,
	linenos,
	numbersep=5pt,
	gobble=0,
	frame=lines,
	framesep=2mm]{python}
[0 -> 2   1 -> 1   2 -> 0
 [-1 -4]  [ 1  0]  [-2 1]
 [ 0  1]  [-3 -1]  [ 1 0]
 Does not extend to link, 
 0 -> 2  1 -> 0   2 -> 1 
 [1  4]  [-3 -1]  [ 1 -1]
 [0 -1]  [ 2  1]  [-2  1]
 Does not extend to link]
\end{minted}
In both of these isometries, the $(1,0)$ slope on the original cusp of $M$ is taken to the $(1,0)$ slope on cusp $2$ of \mfld{L13n9833}.
Since the meridians of the drilled curves have slope $(1,0)$ on their cusps, the action of isometries on the cusps indicates that \mfld{o9\_40504} has surgery descriptions as both \mfld{L13n9833(-2,1)(1,-3)(0,0)} and \mfld{L13n9833(-3,2)(1,-2)(0,0)}.  These surgery descriptions can be confirmed in SnapPy.

Next, using the SnapPy application, entering 
\begin{verbatim}
    Manifold('L13n9833').plink()
\end{verbatim}
or even
\begin{verbatim}
    Manifold('L13n9833').browse()
\end{verbatim}
give diagrams of the link \mfld{L13n9833} along with the indexing of which components correspond to which cusps.  See Figure~\ref{fig:plinkbrowse}.

\begin{figure}
\centering
 \includegraphics[height = 7.5cm]{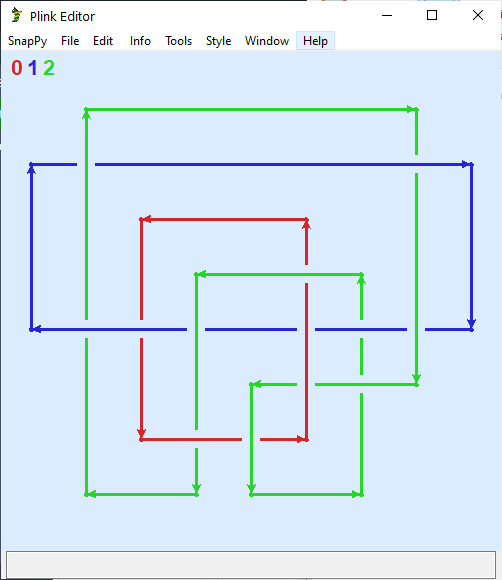} \qquad \includegraphics[height = 7.5cm]{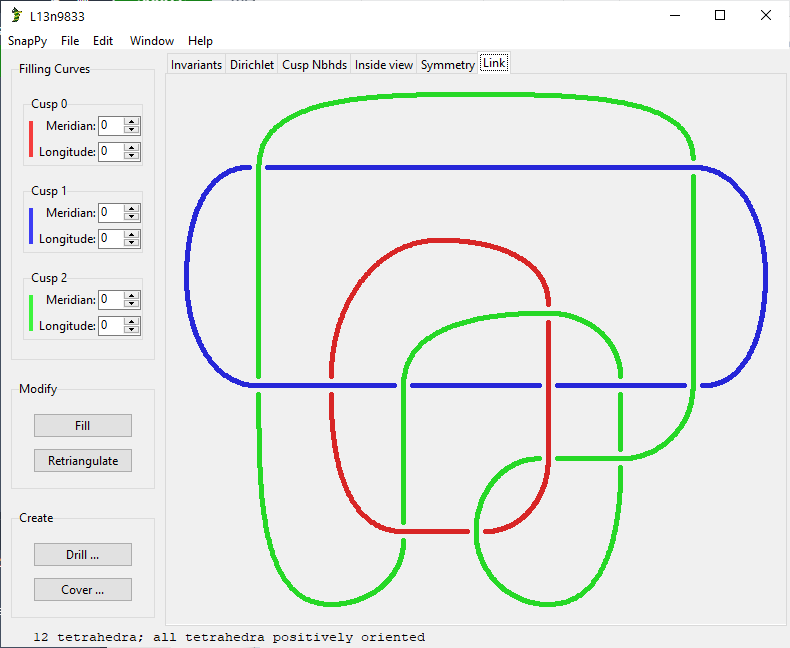} 
\caption{A diagram of the link \mfld{L13n9833} in PLink (left) and in the SnapPy browser (right).}
\label{fig:plinkbrowse}
\end{figure}

Using the diagram, the surgery description may now be drawn in KLO as in Figure~\ref{fig:KLOinitialdiagram} and manipulated through a sequence of Rolfsen twists until a knot diagram is obtained. The key steps on the manipulation are shown in Figure~\ref{fig:KLOtransformations}. 
After producing the diagram in KLO, from the {\em Export} menu of KLO we select {\em diagram for SnapPea} to save a PLink diagram of the knot.

\begin{figure}
    \centering
    \includegraphics[height =7.5cm]{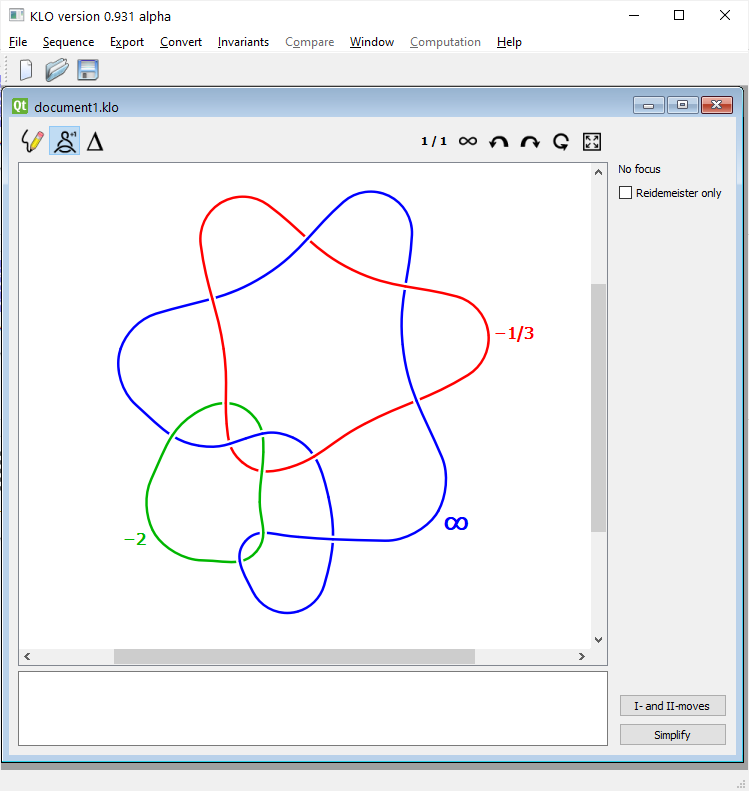}
    \caption{A surgery diagram of the knot \mfld{o9\_40504} on the link \mfld{L13n9833} drawn in KLO.}
    \label{fig:KLOinitialdiagram}
\end{figure}

\begin{figure}
    \centering
    (a)\includegraphics[width = .3\textwidth]{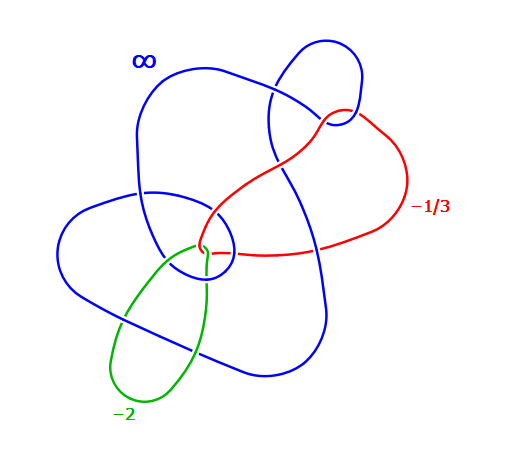} 
    (b)\includegraphics[width = .3\textwidth]{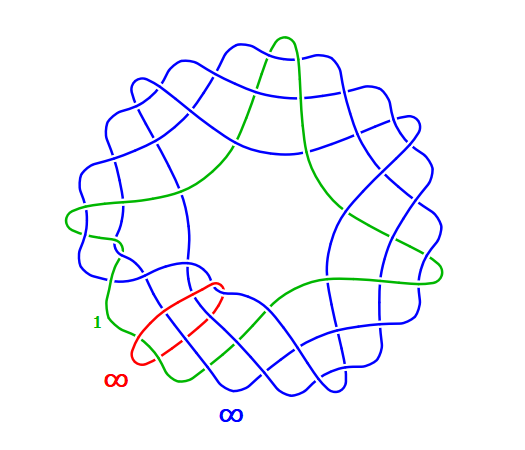} 
    (c)\includegraphics[width = .3\textwidth]{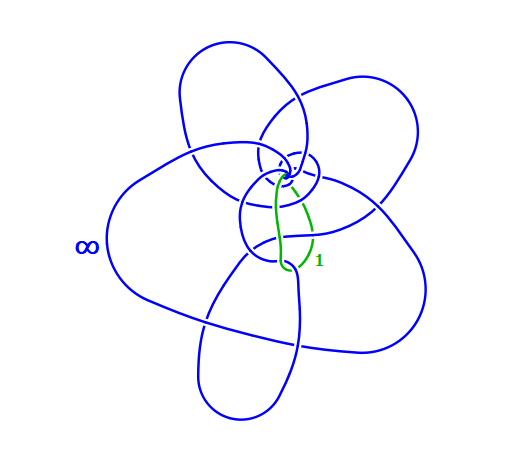} \\
    (d)\includegraphics[width = .3\textwidth]{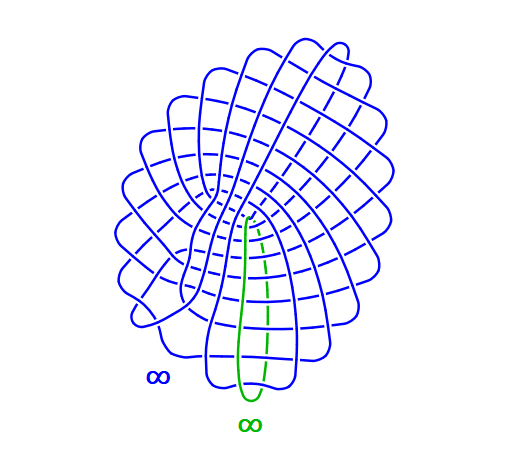} 
    (e)\includegraphics[width = .3\textwidth]{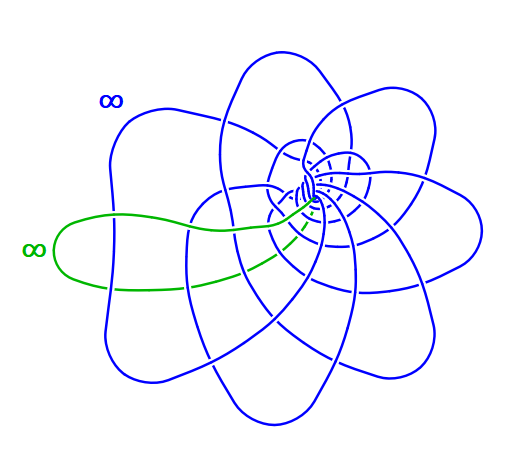} 
    (f)\includegraphics[width = .3\textwidth]{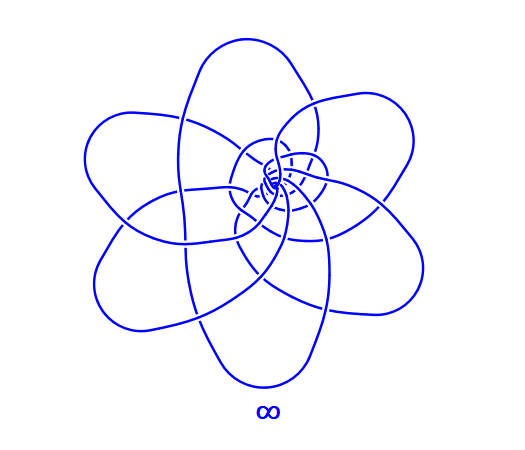} 
    \caption{A sequence of transformations of the surgery diagram of the knot \mfld{o9\_40504} in KLO. (a) The surgery diagram from Figure~\ref{fig:KLOinitialdiagram} has been simplified and prepared for  Rolfsen twists along the unknot with surgery coefficient $-1/3$. (b) The result of $3$ Rolfsen twists is shown.  The unknot which had surgery coefficient $-2$ now has coefficient $1$.  (c) The diagram is now prepared for a Rolfsen twist along the unknot with surgery coefficient $1$.
    (d) The result of the $-1$ Rolfsen twist is shown.  (e) The result with the twisting curve is simplified in KLO.  (f) The twisting curve is discarded and the knot diagram is further simplified in KLO.}
    \label{fig:KLOtransformations}
\end{figure}

With the exported diagram named \verb|diagram.lnk| we may import it back into SnapPy.
Then
\begin{minted}[mathescape,
	linenos,
	numbersep=5pt,
	gobble=0,
	frame=lines,
	framesep=2mm]{python}
K = snappy.Manifold('diagram.lnk').link()
# for orientation correction, may need K = K.mirror()
K.braid_word()  
\end{minted}
returns the list
\begin{quote}
[1, 2, 1, 2, 1, 3, 2, 4, 1, 3, 5, 2, 4, 6, 1, 3, 5, 7, 2, 4, 6, 3, 5, 2, 4, 1, 3, 2, 3, 4, 5, 6, 7, 7, 6, 5, 4, 3, 2, 4, 5, 4, 6, 5, 7, 4, 6, 5, 4, -3, 4]
\end{quote}
that represents the braid word.   One observes that this is not a positive (or negative) braid word.  Performing further simplification within SnapPy with
\begin{minted}[mathescape,
	linenos,
	numbersep=5pt,
	gobble=0,
	frame=lines,
	framesep=2mm]{python}
K.simplify(mode='global',type_III_limit=1000) # K.simplify('global') may be sufficient
K.braid_word() 
\end{minted}
returns the list 
\begin{quote}
[1, 2, 1, 3, 1, 4, 2, 4, 3, 5, 5, 4, 3, 5, 2, 3, 3, 4, 3, 3, 2, 1, 2, 3, 3, 4, 4, 3, 2, 3, 4, 5, 4, 3, 2, 3, 4, 3, 2, 3, 2, 4, 2, 4, 4, 3, 4]
\end{quote}
which represents a positive braid of braid index $6$, realizing the MFW bound.

\section{Data}

Positive braid words that realize the MFW bound are recorded in Table~\ref{tab:braidwords} for knots in $\calA$ and in Tables~\ref{tab:braidwordsT2} and \ref{tab:braidwordsT2-2} for knots in $\calT_2-\calA$.  Also given are their Knot Census name, genus, word length, braid index, and MFW bound.  For each knot the Alexander polynomial $\Delta(t)$ (which one may calculate using SnapPy within Sage) is also given as 
a decreasing sequence of nonnegative integers $n_g>n_{g-1}> \cdots > n_0 = 0$ so that 
\[ \Delta(t) = (-1)^{g} + \sum_{i=1}^{g} (-1)^{g+i} (t^{n_i}+t^{-n_i})\]
and $g$ is the knot genus.
All L-space knots have symmetrized Alexander polynomials $\Delta(t)$ of this special form: the degree of $\Delta(t)$ is the knot genus \cite{genusbounds} and the non-zero coefficients of $\Delta(t)$ alternate between $1$ and $-1$ \cite{Ozsvath2005}.  See also \cite[Corollary 9]{HWbotanygeography}.

The entire list of the manifolds of $\calD$ is split into Tables~\ref{tab:censusLspaceknotsI} and \ref{tab:censusLspaceknotsII}.   The first contains those in the SnapPy census of manifolds assembled with up to 8 ideal tetrahedra.  The second contains those in the SnapPy census assembled with 9 ideal tetrahedra and no fewer.

% \begin{minted}[mathescape,
% 	linenos,
% 	numbersep=5pt,
% 	gobble=0,
% 	frame=lines,
% 	framesep=2mm]{python}
% for mfld in T2list:
%     M=snappy.Manifold(mfld)
%     ap=M.alexander_polynomial()
%     deg=ap.degree()
%     print([(deg/2)-i for i in ap.exponents() if i<=deg/2])
% \end{minted}

\begin{table}
\caption{Each manifold of $\calD$ is listed with a positive braid word whose closure is a knot with the manifold as its complement, the length and index of the braid, the genus of the knot, and the Alexander polynomial of the knot.}
\label{tab:braidwords}
\begin{tabular}{@{}lllll@{}}
\toprule
manifold  &\multicolumn{4}{l}{braid word}\\
knot census & genus & word length  & braid index & MFW bound \\
& \multicolumn{4}{l}{Alexander polynomial}\\
\midrule

\mfld{t12533} 
& \multicolumn{4}{l}{\parbox{0.7\textwidth}{
1, 1, 2, 2, 1, 2, 2, 2, 2, 2, 2, 2, 2, 2, 1, 2, 2, 3, 2, 1, 1, 2, 2, 1, 3, 2, 2
}}\\
\mfld{K8\_290} & 12 & 27  & 4  & 4 \\
& \multicolumn{4}{l}{\parbox{0.7\textwidth}{
$12,11,8,7,5,4,3,2,0$
}}\\
\midrule

\mfld{t12681} &
\multicolumn{4}{l}{\parbox{0.7\textwidth}{
1, 2, 3, 4, 4, 3, 2, 3, 2, 4, 2, 1, 1, 1, 2, 1, 3, 2, 1, 3, 2, 3, 4, 3, 2, 4, 1, 3, 2, 4, 3, 4, 4, 3, 2, 4, 1, 3, 2, 4, 3, 4, 4, 3, 2, 4, 3, 4
}}\\
\mfld{K8\_296} & 22 & 48  & 5 &5 \\
&\multicolumn{4}{l}{\parbox{0.7\textwidth}{
$22,21,17,16,13,12,10,9,7,6,5,4,2,0$
}}\\
\midrule

\mfld{o9\_38928} &
\multicolumn{4}{l}{\parbox{0.7\textwidth}{
1, 2, 1, 2, 3, 2, 4, 2, 3, 4, 4, 5, 4, 3, 5, 2, 3, 3, 2, 1, 3, 2, 3, 3, 2, 4, 2, 3, 2, 4, 3, 5, 4, 3, 5, 3, 2, 4, 1, 2, 2, 3, 3
}}\\
\mfld{K9\_620}& 19 & 43  &  6 & 6 \\
&\multicolumn{4}{l}{\parbox{0.7\textwidth}{
$19, 18, 13, 12, 10, 9, 7, 6, 4, 3, 2, 1, 0$
}}\\
\midrule

\mfld{o9\_39162} &
\multicolumn{4}{l}{\parbox{0.7\textwidth}{
1, 1, 2, 1, 3, 2, 4, 2, 5, 1, 3, 2, 2, 3, 2, 4, 2, 5, 2, 4, 3, 2, 4, 2, 5, 4, 3, 5, 2, 4, 2, 5, 2, 4, 3, 2, 2, 3, 3, 2, 2, 1, 2, 3, 4, 3, 2, 4, 5, 4, 3, 4, 3
}}\\
\mfld{K9\_624} & 24 & 53  & 6 & 6 \\
&\multicolumn{4}{l}{\parbox{0.7\textwidth}{
$24, 23, 18, 17, 15, 14, 11, 10, 9, 8, 6, 5, 3, 1, 0$
}}\\
\midrule

\mfld{o9\_40363} &
\multicolumn{4}{l}{\parbox{0.7\textwidth}{
1, 2, 1, 3, 4, 5, 4, 4, 4, 5, 4, 6, 3, 6, 2, 5, 4, 3, 5, 4, 6, 5, 6, 6, 5, 4, 6, 3, 5, 2, 4, 1, 3, 5, 2, 4, 6, 3, 5, 4, 6, 5, 6, 6, 5, 4, 6, 3, 5, 2, 4, 6, 1, 3, 5, 2, 4, 6, 3, 5, 4, 6, 5, 4, 6, 3, 5, 2, 4, 3, 5, 4
}}\\
\mfld{K9\_674} & 33 & 72 & 7 & 7\\
&\multicolumn{4}{l}{\parbox{0.7\textwidth}{
$33, 32, 26, 25, 21, 20, 18, 17, 14, 13, 12, 11, 9, 8, 6, 4, 2, 1, 0$
}}\\
\midrule

\mfld{o9\_40487} &
\multicolumn{4}{l}{\parbox{0.7\textwidth}{
1, 2, 1, 3, 3, 2, 2, 3, 4, 3, 2, 1, 3, 2, 1, 3, 2, 4, 2, 4, 1, 4, 2, 1, 3, 2, 3, 4, 3, 4, 3, 2
}}\\
\mfld{K9\_679} & 14 & 32 & 5 & 5 \\
&\multicolumn{4}{l}{\parbox{0.7\textwidth}{
14,13,9,8,6,5,3,2,1,0}}\\
\midrule

\mfld{o9\_40504} &
\multicolumn{4}{l}{\parbox{0.7\textwidth}{
1, 1, 2, 1, 3, 4, 3, 4, 3, 5, 4, 3, 5, 2, 4, 1, 3, 1, 2, 1, 3, 4, 5, 4, 3, 5, 4, 3, 5, 5, 5, 4, 3, 2, 1, 3, 4, 5, 4, 4, 5, 4, 3, 2, 4, 3, 4
}}\\
\mfld{K9\_680} & 21 & 47 & 6 & 6\\
&\multicolumn{4}{l}{\parbox{0.7\textwidth}{
$21, 20, 15, 14, 12, 11, 9, 8, 6, 5, 4, 3, 2, 1, 0$
}}\\
\midrule

\mfld{o9\_40582} &
\multicolumn{4}{l}{\parbox{0.7\textwidth}{
1, 2, 2, 3, 2, 2, 3, 4, 3, 2, 1, 2, 3, 2, 4, 4, 3, 3, 2, 1, 3, 3, 3, 2, 2, 3, 4, 3, 2, 2, 1, 2, 2, 3, 2, 3
}}\\
\mfld{K9\_685} & 16 & 36 & 5 & 5 \\
&\multicolumn{4}{l}{\parbox{0.7\textwidth}{
$16, 15, 11, 10, 7, 6, 5, 4, 2, 0$
}}\\
\midrule

\mfld{o9\_42675} &
\multicolumn{4}{l}{\parbox{0.7\textwidth}{
1, 2, 1, 3, 2, 4, 2, 3, 2, 4, 2, 3, 2, 3, 2, 1, 2, 3, 3, 4, 4, 3, 3, 4, 3, 3, 2, 1, 3, 2, 4, 2, 3, 2, 1, 3
}}\\
\mfld{K9\_723} & 16 & 36 & 5 & 5 \\
&\multicolumn{4}{l}{\parbox{0.7\textwidth}{
$16, 15, 11, 10, 8, 7, 5, 4, 3, 2, 1, 0$
}}\\
\bottomrule
\end{tabular}
\end{table}

\begin{table}
	\caption{Each manifold of $\calT_2$ is listed with a positive braid word whose closure is a knot with the manifold as its complement, the length and index of the braid, the genus of the knot, and the Alexander polynomial of the knot.}
	\label{tab:braidwordsT2}
	\begin{tabular}{@{}lllll@{}}
		\toprule
		manifold  &\multicolumn{4}{l}{braid word}\\
		knot census& genus & word length  & braid index & MFW bound \\
		& \multicolumn{4}{l}{Alexander polynomial}\\
		\midrule
		
		\mfld{t09284} 
		& \multicolumn{4}{l}{\parbox{0.7\textwidth}{
				1, 2, 1, 3, 1, 2, 1, 1, 1, 2, 1, 3, 2, 3, 2, 1, 2, 1, 2, 1, 2, 2, 3
			}}\\
		\mfld{K8\_186} & 10 & 23  & 4  & 4 \\
		& \multicolumn{4}{l}{\parbox{0.7\textwidth}{
				$10, 9, 6, 5, 4, 3, 1, 0$
		}}\\
		\midrule

		\mfld{t09450} 
		& \multicolumn{4}{l}{\parbox{0.7\textwidth}{
				1, 2, 3, 4, 5, 5, 4, 3, 5, 4, 3, 5, 2, 1, 3, 2, 4, 3, 5, 2, 3, 2, 4, 3, 5, 3, 2, 3, 3, 3, 3, 4, 5, 4, 3, 4, 4, 3, 5, 4, 3, 3, 2, 3, 4, 3, 5, 3, 5, 4, 3, 2, 3, 2, 4, 3, 4, 4, 5
		}}\\
		\mfld{K8\_189} & 27 & 59  & 6  & 6 \\
		& \multicolumn{4}{l}{\parbox{0.7\textwidth}{
		        $27, 26, 21, 20, 17, 16, 13, 12, 11, 10, 7, 5, 3, 2, 1, 0$
		}}\\
		\midrule

		\mfld{t09633} 
		& \multicolumn{4}{l}{\parbox{0.7\textwidth}{
				1, 2, 1, 2, 3, 2, 4, 1, 3, 2, 2, 1, 2, 3, 4, 3, 3, 2, 3, 4, 3, 2, 2, 1, 3, 3, 2, 3, 3, 2, 4, 3, 3, 2, 3, 4, 3, 2, 1, 2
		}}\\
		\mfld{K8\_195} & 18 & 40  & 5  & 5 \\
		& \multicolumn{4}{l}{\parbox{0.7\textwidth}{
		   $18,17,13,12,10,9,6,4,2,1,0$
		}}\\
		\midrule

		\mfld{t10496} 
		& \multicolumn{4}{l}{\parbox{0.7\textwidth}{
				1, 2, 3, 3, 3, 3, 2, 2, 1, 3, 3, 2, 3, 3, 2, 1, 3, 2, 3, 3, 2, 1, 3, 2, 3, 2, 3
		}}\\
		\mfld{K8\_220} & 12 &  27  & 4  & 4 \\
		& \multicolumn{4}{l}{\parbox{0.7\textwidth}{
		$12, 11, 8, 7, 5, 4, 2, 1, 0$
		}}\\
		\midrule
		
		\mfld{o9\_28751} 
		& \multicolumn{4}{l}{\parbox{0.7\textwidth}{
				1, 2, 1, 3, 4, 5, 4, 6, 3, 5, 7, 4, 6, 5, 7, 4, 6, 5, 6, 7, 6, 5, 4, 6, 5, 4, 6, 3, 2, 4, 1, 3, 2, 1, 3, 1, 4, 5, 4, 6, 3, 4, 5, 4, 3, 5, 4, 6, 3, 5, 2, 6, 1, 7, 6, 5, 5, 4, 3, 5, 2, 4, 1, 3, 5, 2, 4, 6, 5, 7, 4, 5, 6, 5, 5, 4, 3, 2, 1, 3, 2, 4, 3, 5, 2, 4, 1, 3, 5, 2, 4, 6, 1, 3, 5, 7, 4, 6, 3, 5, 2, 4, 3, 2, 1, 3, 2, 4, 3
		}}\\
		\mfld{K9\_412} & 51 & 109  & 8  & 8 \\
		& \multicolumn{4}{l}{\parbox{0.7\textwidth}{
		51, 50, 43, 42, 37, 36, 33, 32, 29, 28, 25, 24, 23, 22, 19, 18, 16, 14, 11, 10, 9, 8, 7, 6, 5, 4, 2, 0
		}}\\
		\midrule
			
		\mfld{o9\_29751} 
		& \multicolumn{4}{l}{\parbox{0.7\textwidth}{
				1, 2, 3, 4, 5, 4, 5, 4, 6, 3, 5, 2, 4, 6, 1, 3, 5, 2, 4, 6, 3, 5, 4, 6, 3, 5, 2, 4, 6, 1, 3, 5, 2, 4, 6, 3, 5, 4, 6, 6, 5, 6, 6, 6, 5, 4, 6, 3, 5, 2, 4, 6, 3, 5, 4, 6, 5, 6, 5, 6
		}}\\
		\mfld{K9\_429} &27 & 60  & 7  & 7 \\
		& \multicolumn{4}{l}{\parbox{0.7\textwidth}{
		$27, 26, 20, 19, 17, 16, 13, 12, 10, 9, 7, 6, 5, 4, 3, 2, 0$
		}}\\
		%\midrule

		\bottomrule
	\end{tabular}
\end{table}

\begin{table}	
	\caption{Table~\ref{tab:braidwordsT2} continued}
	\label{tab:braidwordsT2-2}
\begin{tabular}{@{}lllll@{}}
	\toprule
	manifold  &\multicolumn{4}{l}{braid word}\\
		knot census& genus & word length  & braid index & MFW bound \\
	& \multicolumn{4}{l}{Alexander polynomial}\\
	\midrule

		\mfld{o9\_32314} 
		& \multicolumn{4}{l}{\parbox{0.7\textwidth}{
				1, 2, 1, 3, 2, 4, 1, 3, 5, 2, 4, 6, 1, 3, 5, 2, 4, 1, 3, 2, 1, 1, 2, 3, 4, 5, 6, 6, 5, 4, 3, 2, 1, 1, 1, 2, 1, 3, 2, 4, 1, 3, 5, 2, 4, 6, 1, 3, 5, 2, 4, 1, 3, 5, 2, 4, 6, 1, 3, 6, 2, 6, 1, 5, 4, 3, 2, 1
		}}\\
		\mfld{K9\_481} & 31 & 68  & 7  & 7 \\
		& \multicolumn{4}{l}{\parbox{0.7\textwidth}{
				$31, 30, 24, 23, 21, 20, 16, 15, 14, 13, 11, 10, 8, 7, 6, 5, 4, 3, 1, 0$
		}}\\
		\midrule
		
		\mfld{o9\_33380} 
		& \multicolumn{4}{l}{\parbox{0.7\textwidth}{
				1, 2, 1, 2, 1, 3, 1, 4, 2, 4, 1, 3, 2, 4, 1, 3, 2, 4, 3, 4, 4, 3, 2, 4, 1, 3, 2, 4, 3, 4, 4, 3, 2, 4, 1, 3, 2, 4, 3, 4, 4, 3, 2, 4
		}}\\
		\mfld{K9\_497} &20 & 44  & 5  & 5 \\
		& \multicolumn{4}{l}{\parbox{0.7\textwidth}{
		        $20, 19, 15, 14, 11, 10, 9, 8, 6, 5, 4, 3, 1, 0$
		}}\\
		\midrule
		
		\mfld{o9\_33944} 
		& \multicolumn{4}{l}{\parbox{0.7\textwidth}{
				1, 2, 1, 3, 2, 4, 1, 3, 2, 4, 3, 2, 4, 1, 3, 2, 4, 3, 4, 4, 3, 2, 4, 1, 3, 2, 4, 3, 4, 4, 3, 2, 4, 1, 3, 2, 4, 3, 4, 4, 3, 2, 4, 1, 3, 2, 4, 3, 4, 4, 3, 2, 3, 4, 4, 3
		}}\\
		\mfld{K9\_511} & 26 & 56  & 5  & 5 \\
		& \multicolumn{4}{l}{\parbox{0.7\textwidth}{
				$26, 25, 21, 20, 16, 15, 12, 11, 10, 9, 7, 6, 5, 4, 2, 1, 0$
		}}\\
		\midrule
		
		\mfld{o9\_33959} 
		& \multicolumn{4}{l}{\parbox{0.7\textwidth}{
				1, 2, 3, 2, 4, 3, 5, 3, 5, 4, 3, 5, 2, 5, 2, 4, 4, 3, 2, 4, 1, 3, 5, 2, 5, 3, 6, 4, 3, 5, 2, 4, 3, 2, 2, 1, 2, 2, 3, 3, 4, 3, 5, 2, 6, 3, 4, 5, 4, 3, 2, 2, 1, 3, 2, 4, 1, 5, 4, 3, 2, 4, 3, 3, 2, 1, 3, 2, 3, 4, 4, 3, 4, 5, 4, 3, 2, 1, 3, 3, 4, 3, 5, 4
		}}\\
		\mfld{K9\_513} & 39 & 84 & 7 & 7 \\
		& \multicolumn{4}{l}{\parbox{0.7\textwidth}{
				$39, 38, 32, 31, 27, 26, 23, 22, 20, 19, 16, 15, 14, 13, 11, 10, 8, 6, 4, 3, 2, 1, 0$
		}}\\
		\midrule
		
		\mfld{o9\_34409} 
		& \multicolumn{4}{l}{\parbox{0.7\textwidth}{
				1, 2, 1, 1, 2, 1, 1, 2, 1, 1, 2, 1, 1, 2, 1, 1, 2, 1, 1, 2, 1, 1, 2, 1, 1, 2, 3, 2, 1, 1, 2, 2, 3
		}}\\
		\mfld{K9\_524} & 15 & 33  & 4  & 4 \\
		& \multicolumn{4}{l}{\parbox{0.7\textwidth}{
				$15, 14, 11, 10, 8, 7, 5, 4, 2, 1, 0$
		}}\\
		\midrule
		
		\mfld{o9\_36380} 
		& \multicolumn{4}{l}{\parbox{0.7\textwidth}{
				1, 2, 1, 3, 2, 4, 1, 3, 5, 2, 4, 6, 1, 3, 5, 7, 2, 4, 6, 8, 1, 3, 5, 7, 9, 2, 4, 6, 8, 1, 3, 5, 7, 2, 4, 6, 1, 3, 5, 2, 4, 1, 3, 2, 1, 1, 2, 1, 3, 2, 4, 1, 3, 5, 2, 4, 6, 1, 3, 5, 7, 2, 4, 6, 8, 1, 3, 5, 7, 9, 2, 4, 6, 8, 1, 3, 5, 7, 2, 4, 6, 1, 3, 5, 2, 4, 1, 3, 4, 5, 6, 7, 8, 9, 9, 8, 7, 9, 6, 8, 5, 7, 9, 4, 6, 9, 3, 5, 8, 2, 4, 7, 1, 3, 6, 8, 2, 5, 7, 9, 4, 8, 3, 7, 6, 5, 4
		}}\\
		\mfld{K9\_565} & 59 & 127  & 10  & 10 \\
		& \multicolumn{4}{l}{\parbox{0.7\textwidth}{
				59, 58, 49, 48, 45, 44, 39, 38, 35, 34, 31, 30, 28, 27, 25, 24, 21, 20, 19, 18, 17, 16, 14, 13, 11, 10, 9, 8, 7, 6, 5, 4, 3, 2, 0
		}}\\
		\midrule
		
		\mfld{o9\_40026} 
		& \multicolumn{4}{l}{\parbox{0.7\textwidth}{
				1, 2, 3, 4, 4, 3, 2, 4, 1, 4, 3, 2, 3, 4, 3, 4, 3, 4, 4, 3, 2, 3, 2, 1, 2, 2, 3, 2, 2, 3, 3, 2, 3, 4, 3, 2, 4, 2, 4, 1, 2, 3
		}}\\
		\mfld{K9\_656} & 19 & 42  & 5  & 5 \\
		& \multicolumn{4}{l}{\parbox{0.7\textwidth}{
		19,18,14,13,11,10,7,6,5,4,3,2,0}}\\
		%\midrule

\bottomrule
\end{tabular}
\end{table}

\section{Acknowledgements}
This article began as a record of the results of an ``office hour'' session on the use of SnapPy \cite{snappy} and KLO \cite{KLO} held at the ICERM workshop Perspectives on Dehn Surgery,  July 15--19, 2019.  We thank ICERM (the Institute for Computational and Experimental Research in Mathematics in Providence, RI) for the productive environment where this work could be carried out and Nathan Dunfield for his interest and assistance. Furthermore we thank all of SnapPy's creators/contributors and KLO's Frank Swenton for producing and maintaining these ever-useful programs.   We also thank the anonymous referee for useful suggestions and the recommendation to expand the scope of this work.

As a part of the ICERM workshop, this work is supported in part by the National Science Foundation under Grant No. DMS-1439786 and by the NSF CAREER Award DMS-1455132.
KLB was partially supported by a grant from the Simons Foundation (grant \#523883  to Kenneth L.\ Baker).
KM was partially supported by NSF RTG grant DMS-1344991 and by the Simons Foundation.
SO was partially supported by Turkish Academy of Sciences T\"{U}BA-GEB\.{I}P Award.
ST’s attendance at the workshop was supported by the Dartmouth Department of Mathematics. 
AW's attendance at the workshop was supported by the grant NSF Career Award DMS-1455132.

\begin{table}[]
    \caption{The members of $\calD$ in the SnapPy census of manifolds assembled from at most 8 ideal tetrahedra.}
    %The list $\calD$ of the 632 L-space knot complements in the SnapPy census.  Included are the two manifolds \mfld{o9\_30150} and \mfld{o9\_31440} whose status is unconfirmed. Part I: Those of at most 8 tetrahedra.}
    \label{tab:censusLspaceknotsI}
{\small 
\mfld{m016},
\mfld{m071},
\mfld{m082},
\mfld{m103},
\mfld{m118},
\mfld{m144},
\mfld{m194},
\mfld{m198},
\mfld{m211},
\mfld{m223},
\mfld{m239},
\mfld{m240},
\mfld{m270},
\mfld{m276},
\mfld{m281},
\mfld{s042},
\mfld{s068},
\mfld{s086},
\mfld{s104},
\mfld{s114},
\mfld{s294},
\mfld{s301},
\mfld{s308},
\mfld{s336},
\mfld{s344},
\mfld{s346},
\mfld{s367},
\mfld{s369},
\mfld{s384},
\mfld{s407},
\mfld{s560},
\mfld{s582},
\mfld{s652},
\mfld{s665},
\mfld{s682},
\mfld{s684},
\mfld{s769},
\mfld{s800},
\mfld{s849},
\mfld{v0082}, %k73  T(5, 16, 2, 1)
\mfld{v0114}, %k74  T(5, −17, 3, −1)
\mfld{v0165}, %k75  T(3, −17, 2, −1)
\mfld{v0220}, %k76  T(7, −17, 2, 1)
\mfld{v0223},
\mfld{v0249},
\mfld{v0319},
\mfld{v0330},
\mfld{v0398},
\mfld{v0407},
\mfld{v0424},
\mfld{v0434},
\mfld{v0497},
\mfld{v0545},
\mfld{v0554},
\mfld{v0570},
\mfld{v0573},
\mfld{v0707},
\mfld{v0709},
\mfld{v0715},
\mfld{v0740},
\mfld{v0741},
\mfld{v0759},
\mfld{v0765},
\mfld{v0830},
\mfld{v0847},
\mfld{v0912},
\mfld{v0939},
\mfld{v0945},
\mfld{v0959},
\mfld{v1077},
\mfld{v1109},
\mfld{v1269},
\mfld{v1300},
\mfld{v1359},
\mfld{v1392},
\mfld{v1423},
\mfld{v1425},
\mfld{v1547},
\mfld{v1565},
\mfld{v1620},
\mfld{v1628},
\mfld{v1690},
\mfld{v1709},
\mfld{v1716},
\mfld{v1718},
\mfld{v1728},
\mfld{v1810},
\mfld{v1832},
\mfld{v1839},
\mfld{v1915},
\mfld{v1921},
\mfld{v1940},
\mfld{v1966},
\mfld{v1980},
\mfld{v1986},
\mfld{v2024},
\mfld{v2090},
\mfld{v2215},
\mfld{v2217},
\mfld{v2290},
\mfld{v2325},
\mfld{v2384},
\mfld{v2759},
\mfld{v2871},
\mfld{v2900},
\mfld{v2925},
\mfld{v2930},
\mfld{v3070},
\mfld{v3105},
\mfld{v3234},
\mfld{v3335},
\mfld{v3354},  %k7 123  T(8, 5, 6, 2)
\mfld{v3482},  %k7 126  T(8, −3, 4, −2)
\mfld{t00110},
\mfld{t00146},
\mfld{t00324},
\mfld{t00423},
\mfld{t00434},
\mfld{t00550},
\mfld{t00621},
\mfld{t00729},
\mfld{t00787},
\mfld{t00826},
\mfld{t00855},
\mfld{t00873},
\mfld{t00932},
\mfld{t01033},
\mfld{t01037},
\mfld{t01125},
\mfld{t01216},
\mfld{t01268},
\mfld{t01292},
\mfld{t01318},
\mfld{t01368},
\mfld{t01409},
\mfld{t01422},
\mfld{t01424},
\mfld{t01440},
\mfld{t01598},
\mfld{t01636},
\mfld{t01646},
\mfld{t01690},
\mfld{t01757},
\mfld{t01815},
\mfld{t01834},
\mfld{t01850},
\mfld{t01863},
\mfld{t01949},
\mfld{t01966},
\mfld{t02099},
\mfld{t02104},
\mfld{t02238},
\mfld{t02276},
\mfld{t02378},
\mfld{t02398},
\mfld{t02404},
\mfld{t02470},
\mfld{t02537},
\mfld{t02567},
\mfld{t02639},
\mfld{t03106},
\mfld{t03566},
\mfld{t03607},
\mfld{t03709},
\mfld{t03710},
\mfld{t03713},
\mfld{t03781},
\mfld{t03843},
\mfld{t03864},
\mfld{t03956},
\mfld{t03979},
\mfld{t04003},
\mfld{t04019},
\mfld{t04102},
\mfld{t04180},
\mfld{t04228},
\mfld{t04244},
\mfld{t04382},
\mfld{t04449},
\mfld{t04557},
\mfld{t04721},
\mfld{t04756},
\mfld{t04927},
\mfld{t05118},
\mfld{t05239},
\mfld{t05390},
\mfld{t05425},
\mfld{t05426},
\mfld{t05538},
\mfld{t05564},
\mfld{t05578},
\mfld{t05658},
\mfld{t05663},
\mfld{t05674},
\mfld{t05695},
\mfld{t06001},
\mfld{t06246},
\mfld{t06440},
\mfld{t06463},
\mfld{t06525},
\mfld{t06570},
\mfld{t06573},
\mfld{t06605},
\mfld{t06637},
\mfld{t06715},
\mfld{t06957},
\mfld{t07070},
\mfld{t07104},
\mfld{t07348},
\mfld{t07355},
\mfld{t07412},
\mfld{t07670},
\mfld{t08111},
\mfld{t08114},
\mfld{t08184},
\mfld{t08201},
\mfld{t08267},
\mfld{t08273},
\mfld{t08403},
\mfld{t08532},
\mfld{t08576},
\mfld{t08936},
\mfld{t09016},
\mfld{t09126},
\mfld{t09267},
\mfld{t09284},
\mfld{t09313},
\mfld{t09450},
\mfld{t09455},
\mfld{t09500},
\mfld{t09580},
\mfld{t09633},
\mfld{t09690},
\mfld{t09704},
\mfld{t09847},
\mfld{t09852},
\mfld{t09882},
\mfld{t09912},
\mfld{t09954},
\mfld{t10177},
\mfld{t10188},
\mfld{t10215},
\mfld{t10224},
\mfld{t10230},
\mfld{t10262},
\mfld{t10292},
\mfld{t10462},
\mfld{t10496},
\mfld{t10643},
\mfld{t10681},
\mfld{t10832},
\mfld{t10985},
\mfld{t11198},
\mfld{t11376},
\mfld{t11548},
\mfld{t11556},
\mfld{t11852},
\mfld{t11887},
\mfld{t11909},
\mfld{t12288},
\mfld{t12533},
\mfld{t12681},
\mfld{t12753}
}
\end{table}
\begin{table}[]
    \caption{The members of $\calD$ in the SnapPy census of manifolds assembled from 9 ideal tetrahedra.}
    \label{tab:censusLspaceknotsII}
{\small 
\mfld{o9\_00133},
\mfld{o9\_00168},
\mfld{o9\_00644},
\mfld{o9\_00797},
\mfld{o9\_00815},
\mfld{o9\_01079},
\mfld{o9\_01175},
\mfld{o9\_01436},
\mfld{o9\_01496},
\mfld{o9\_01584},
\mfld{o9\_01621},
\mfld{o9\_01680},
\mfld{o9\_01765},
\mfld{o9\_01936},
\mfld{o9\_01953},
\mfld{o9\_01955},
\mfld{o9\_02255},
\mfld{o9\_02340},
\mfld{o9\_02350},
\mfld{o9\_02383},
\mfld{o9\_02386},
\mfld{o9\_02655},
\mfld{o9\_02696},
\mfld{o9\_02706},
\mfld{o9\_02735},
\mfld{o9\_02772},
\mfld{o9\_02786},
\mfld{o9\_02794},
\mfld{o9\_02909},
\mfld{o9\_03032},
\mfld{o9\_03108},
\mfld{o9\_03118},
\mfld{o9\_03133},
\mfld{o9\_03149},
\mfld{o9\_03162},
\mfld{o9\_03188},
\mfld{o9\_03288},
\mfld{o9\_03313},
\mfld{o9\_03412},
\mfld{o9\_03526},
\mfld{o9\_03586},
\mfld{o9\_03622},
\mfld{o9\_03802},
\mfld{o9\_03833},
\mfld{o9\_03932},
\mfld{o9\_04054},
\mfld{o9\_04060},
\mfld{o9\_04106},
\mfld{o9\_04205},
\mfld{o9\_04245},
\mfld{o9\_04269},
\mfld{o9\_04313},
\mfld{o9\_04431},
\mfld{o9\_04435},
\mfld{o9\_04438},
\mfld{o9\_04938},
\mfld{o9\_05021},
\mfld{o9\_05177},
\mfld{o9\_05229},
\mfld{o9\_05287},
\mfld{o9\_05357},
\mfld{o9\_05426},
\mfld{o9\_05483},
\mfld{o9\_05562},
\mfld{o9\_05618},
\mfld{o9\_05860},
\mfld{o9\_05970},
\mfld{o9\_06060},
\mfld{o9\_06128},
\mfld{o9\_06154},
\mfld{o9\_06248},
\mfld{o9\_06301},
\mfld{o9\_06956},
\mfld{o9\_07044},
\mfld{o9\_07152},
\mfld{o9\_07167},
\mfld{o9\_07195},
\mfld{o9\_07401},
\mfld{o9\_07790},
\mfld{o9\_07893},
\mfld{o9\_07943},
\mfld{o9\_07945},
\mfld{o9\_08006},
\mfld{o9\_08042},
\mfld{o9\_08224},
\mfld{o9\_08302},
\mfld{o9\_08402},
\mfld{o9\_08477},
\mfld{o9\_08497},
\mfld{o9\_08647},
\mfld{o9\_08765},
\mfld{o9\_08771},
\mfld{o9\_08776},
\mfld{o9\_08828},
\mfld{o9\_08831},
\mfld{o9\_08852},
\mfld{o9\_08875},
\mfld{o9\_09052},
\mfld{o9\_09213},
\mfld{o9\_09271},
\mfld{o9\_09372},
\mfld{o9\_09465},
\mfld{o9\_09731},
\mfld{o9\_09808},
\mfld{o9\_10020},
\mfld{o9\_10192},
\mfld{o9\_10213},
\mfld{o9\_10696},
\mfld{o9\_11002},
\mfld{o9\_11100},
\mfld{o9\_11248},
\mfld{o9\_11467},
\mfld{o9\_11537},
\mfld{o9\_11541},
\mfld{o9\_11556},
\mfld{o9\_11560},
\mfld{o9\_11570},
\mfld{o9\_11658},
\mfld{o9\_11685},
\mfld{o9\_11795},
\mfld{o9\_11845},
\mfld{o9\_11999},
\mfld{o9\_12079},
\mfld{o9\_12144},
\mfld{o9\_12230},
\mfld{o9\_12253},
\mfld{o9\_12412},
\mfld{o9\_12459},
\mfld{o9\_12477},
\mfld{o9\_12519},
\mfld{o9\_12693},
\mfld{o9\_12736},
\mfld{o9\_12757},
\mfld{o9\_12873},
\mfld{o9\_12892},
\mfld{o9\_12919},
\mfld{o9\_12971},
\mfld{o9\_13052},
\mfld{o9\_13054},
\mfld{o9\_13056},
\mfld{o9\_13125},
\mfld{o9\_13182},
\mfld{o9\_13188},
\mfld{o9\_13400},
\mfld{o9\_13403},
\mfld{o9\_13433},
\mfld{o9\_13508},
\mfld{o9\_13537},
\mfld{o9\_13604},
\mfld{o9\_13639},
\mfld{o9\_13649},
\mfld{o9\_13666},
\mfld{o9\_13720},
\mfld{o9\_13952},
\mfld{o9\_14018},
\mfld{o9\_14079},
\mfld{o9\_14108},
\mfld{o9\_14136},
\mfld{o9\_14359},
\mfld{o9\_14364},
\mfld{o9\_14376},
\mfld{o9\_14495},
\mfld{o9\_14599},
\mfld{o9\_14716},
\mfld{o9\_14831},
\mfld{o9\_14974},
\mfld{o9\_15506},
\mfld{o9\_15633},
\mfld{o9\_15808},
\mfld{o9\_15997},
\mfld{o9\_16065},
\mfld{o9\_16141},
\mfld{o9\_16157},
\mfld{o9\_16181},
\mfld{o9\_16319},
\mfld{o9\_16356},
\mfld{o9\_16431},
\mfld{o9\_16514},
\mfld{o9\_16527},
\mfld{o9\_16642},
\mfld{o9\_16685},
\mfld{o9\_16748},
\mfld{o9\_16920},
\mfld{o9\_17382},
\mfld{o9\_17450},
\mfld{o9\_17646},
\mfld{o9\_18007},
\mfld{o9\_18209},
\mfld{o9\_18341},
\mfld{o9\_18633},
\mfld{o9\_18646},
\mfld{o9\_18813},
\mfld{o9\_19130},
\mfld{o9\_19247},
\mfld{o9\_19364},
\mfld{o9\_19396},
\mfld{o9\_19645},
\mfld{o9\_19724},
\mfld{o9\_20029},
\mfld{o9\_20219},
\mfld{o9\_20305},
\mfld{o9\_20364},
\mfld{o9\_20472},
\mfld{o9\_21195},
\mfld{o9\_21496},
\mfld{o9\_21513},
\mfld{o9\_21620},
\mfld{o9\_21893},
\mfld{o9\_21918},
\mfld{o9\_22129},
\mfld{o9\_22252},
\mfld{o9\_22477},
\mfld{o9\_22607},
\mfld{o9\_22663},
\mfld{o9\_22698},
\mfld{o9\_22925},
\mfld{o9\_23023},
\mfld{o9\_23032},
\mfld{o9\_23179},
\mfld{o9\_23263},
\mfld{o9\_23461},
\mfld{o9\_23660},
\mfld{o9\_23723},
\mfld{o9\_23955},
\mfld{o9\_23961},
\mfld{o9\_23971},
\mfld{o9\_23977},
\mfld{o9\_24069},
\mfld{o9\_24126},
\mfld{o9\_24149},
\mfld{o9\_24183},
\mfld{o9\_24290},
\mfld{o9\_24401},
\mfld{o9\_24407},
\mfld{o9\_24534},
\mfld{o9\_24592},
\mfld{o9\_24779},
\mfld{o9\_24886},
\mfld{o9\_24889},
\mfld{o9\_24946},
\mfld{o9\_25110},
\mfld{o9\_25199},
\mfld{o9\_25341},
\mfld{o9\_25444},
\mfld{o9\_25595},
\mfld{o9\_25709},
\mfld{o9\_25832},
\mfld{o9\_26141},
\mfld{o9\_26471},
\mfld{o9\_26570},
\mfld{o9\_26604},
\mfld{o9\_26767},
\mfld{o9\_26791},
\mfld{o9\_27107},
\mfld{o9\_27155},
\mfld{o9\_27261},
\mfld{o9\_27371},
\mfld{o9\_27392},
\mfld{o9\_27429},
\mfld{o9\_27480},
\mfld{o9\_27737},
\mfld{o9\_27767},
\mfld{o9\_28113},
\mfld{o9\_28153},
\mfld{o9\_28284},
\mfld{o9\_28529},
\mfld{o9\_28592},
\mfld{o9\_28746},
\mfld{o9\_28751},
\mfld{o9\_28810},
\mfld{o9\_29048},
\mfld{o9\_29246},
\mfld{o9\_29436},
\mfld{o9\_29529},
\mfld{o9\_29551},
\mfld{o9\_29648},
\mfld{o9\_29751},
\mfld{o9\_29766},
\mfld{o9\_30142},
\mfld{o9\_30150}, % uncertain L-space knot status now confirmed
\mfld{o9\_30375},
\mfld{o9\_30634},
\mfld{o9\_30650},
\mfld{o9\_30721},
\mfld{o9\_30790},
\mfld{o9\_31165},
\mfld{o9\_31267},
\mfld{o9\_31321},
\mfld{o9\_31440}, % uncertain L-space knot status now confirmed
\mfld{o9\_31481},
\mfld{o9\_32044},
\mfld{o9\_32065},
\mfld{o9\_32132},
\mfld{o9\_32150},
\mfld{o9\_32257},
\mfld{o9\_32314},
\mfld{o9\_32471},
\mfld{o9\_32588},
\mfld{o9\_32964},
\mfld{o9\_33189},
\mfld{o9\_33284},
\mfld{o9\_33380},
\mfld{o9\_33430},
\mfld{o9\_33486},
\mfld{o9\_33526},
\mfld{o9\_33585},
\mfld{o9\_33801},
\mfld{o9\_33944},
\mfld{o9\_33959},
\mfld{o9\_34000},
\mfld{o9\_34403},
\mfld{o9\_34409},
\mfld{o9\_34689},
\mfld{o9\_35320},
\mfld{o9\_35549},
\mfld{o9\_35666},
\mfld{o9\_35682},
\mfld{o9\_35720},
\mfld{o9\_35736},
\mfld{o9\_35772},
\mfld{o9\_35928},
\mfld{o9\_36114},
\mfld{o9\_36250},
\mfld{o9\_36380},
\mfld{o9\_36544},
\mfld{o9\_36809},
\mfld{o9\_36958},
\mfld{o9\_37050},
\mfld{o9\_37291},
\mfld{o9\_37482},
\mfld{o9\_37551},
\mfld{o9\_37685},
\mfld{o9\_37751},
\mfld{o9\_37754},
\mfld{o9\_37851},
\mfld{o9\_37941},
\mfld{o9\_38287},
\mfld{o9\_38679},
\mfld{o9\_38811},
\mfld{o9\_38928},
\mfld{o9\_38989},
\mfld{o9\_39162},
\mfld{o9\_39394},
\mfld{o9\_39451},
\mfld{o9\_39521},
\mfld{o9\_39606},
\mfld{o9\_39608},
\mfld{o9\_39859},
\mfld{o9\_39879},
\mfld{o9\_39981},
\mfld{o9\_40026},
\mfld{o9\_40052},
\mfld{o9\_40075},
\mfld{o9\_40179},
\mfld{o9\_40363},
\mfld{o9\_40487},
\mfld{o9\_40504},
\mfld{o9\_40582},
\mfld{o9\_41372},
\mfld{o9\_42224},
\mfld{o9\_42493},
\mfld{o9\_42675},
\mfld{o9\_42961},
\mfld{o9\_43001},
\mfld{o9\_43679},
\mfld{o9\_43750},
\mfld{o9\_43857},
\mfld{o9\_43953},
\mfld{o9\_44054}
}
\end{table}

\let\MRhref\undefined
\bibliographystyle{hamsalpha}

\bibliography{sources.bib}

\newcommand{\etalchar}[1]{$^{#1}$}
\providecommand{\bysame}{\leavevmode\hbox to3em{\hrulefill}\thinspace}
\providecommand{\MR}{\relax\ifhmode\unskip\space\fi MR }
% \MRhref is called by the amsart/book/proc definition of \MR.
\providecommand{\MRhref}[2]{%
  \href{http://www.ams.org/mathscinet-getitem?mr=#1}{#2}
}
\providecommand{\href}[2]{#2}
\begin{thebibliography}{BBP{\etalchar{+}}20}

\bibitem[Auc14]{Auckly}
David Auckly, \emph{Two-fold branched covers}, J. Knot Theory Ramifications
  \textbf{23} (2014), no.~3, 1430001, 29. \MR{3200493}

\bibitem[BBP{\etalchar{+}}20]{regina}
Benjamin~A. Burton, Ryan Budney, William Pettersson, et~al., \emph{Regina:
  Software for low-dimensional topology}, {\tt http://\allowbreak
  regina-normal.\allowbreak github.\allowbreak io/}, 1999--2020.

\bibitem[Ber18]{heegaard}
John Berge, \emph{Heegaard3: a program for studying {H}eegaard splittings},
  2018, \url{http://bitbucket.org/t3m/heegaard3}.

\bibitem[BKM]{BKM}
Kenneth~L. Baker, Marc Kegel, and Duncan McCoy, \emph{The search for
  alternating and quasi-alternating surgeries on asymmetric knots}, In
  progress, December 2020.

\bibitem[BL20]{Baker2017}
Kenneth~L. Baker and John Luecke, \emph{Asymmetric {L}--space knots}, Geom.
  Topol. \textbf{24} (2020), no.~5, 2287--2359,
  \href{http://arxiv.org/abs/1710.01655v1}{arXiv:1710.01655v1 [math.GT]}.
  \MR{4194294}

\bibitem[BM15]{Baker2015}
Kenneth~L. Baker and Kimihiko Motegi, \emph{Twist families of {L}-space knots,
  their genera, and {S}eifert surgeries},
  \href{http://arxiv.org/abs/1506.04455v2}{arXiv:1506.04455v2 [math.GT]}, to
  appear in Comm. Anal. Geom.

\bibitem[BM18]{Baker2014}
Kenneth~L. Baker and Allison~H. Moore, \emph{Montesinos knots, {H}opf
  plumbings, and {L}-space surgeries}, J. Math. Soc. Japan \textbf{70} (2018),
  no.~1, 95--110, \href{http://arxiv.org/abs/1404.7585v1}{arXiv:1404.7585v1
  [math.GT]}. \MR{3750269}

\bibitem[CDGW]{snappy}
Marc Culler, Nathan~M. Dunfield, Matthias Goerner, and Jeffrey~R. Weeks,
  \emph{Snap{P}y, a computer program for studying the geometry and topology of
  3-manifolds}, (29/07/2019), \url{http://snappy.computop.org}.

\bibitem[CDW99]{CDW}
Patrick~J. Callahan, John~C. Dean, and Jeffrey~R. Weeks, \emph{The simplest
  hyperbolic knots}, J. Knot Theory Ramifications \textbf{8} (1999), no.~3,
  279--297. \MR{1691433}

\bibitem[CKM14]{CKM}
Abhijit Champanerkar, Ilya Kofman, and Timothy Mullen, \emph{The 500 simplest
  hyperbolic knots}, J. Knot Theory Ramifications \textbf{23} (2014), no.~12,
  1450055, 34. \MR{3298204}

\bibitem[CKP04]{CKP}
Abhijit Champanerkar, Ilya Kofman, and Eric Patterson, \emph{The next simplest
  hyperbolic knots}, J. Knot Theory Ramifications \textbf{13} (2004), no.~7,
  965--987. \MR{2101238}

\bibitem[Dea03]{Dean2003}
John~C. Dean, \emph{Small {S}eifert-fibered {D}ehn surgery on hyperbolic
  knots}, Algebr. Geom. Topol. \textbf{3} (2003), 435--472. \MR{1997325}

\bibitem[Dun18]{Dunfield2018}
Nathan~M. Dunfield, \emph{A census of exceptional {D}ehn fillings},
  \href{http://arxiv.org/abs/1812.11940v1}{arXiv:1812.11940v1 [math.GT]}.

\bibitem[Dun19]{Dunfield2019}
\bysame, \emph{Floer homology, group orderability, and taut foliations of
  hyperbolic 3-manifolds},
  \href{http://arxiv.org/abs/1904.04628v1}{arXiv:1904.04628v1 [math.GT]}.

\bibitem[FW87]{FranksWilliamsBraid}
John Franks and R.~F. Williams, \emph{Braids and the {J}ones polynomial},
  Trans. Amer. Math. Soc. \textbf{303} (1987), no.~1, 97--108. \MR{896009}

\bibitem[Ghi08]{Ghiggini2008}
Paolo Ghiggini, \emph{Knot {F}loer homology detects genus-one fibred knots},
  Amer. J. Math. \textbf{130} (2008), no.~5, 1151--1169,
  \href{http://arxiv.org/abs/math/0603445v3}{arXiv:math/0603445v3 [math.GT]}.
  \MR{2450204}

\bibitem[Gro20]{kitchensink}
The~Computop Group, \emph{Computop {S}age{M}ath {D}ocker {I}mage}, 2020,
  \url{https://hub.docker.com/r/computop/sage}.

\bibitem[Hed05]{Hedden2005}
Matthew Hedden, \emph{On knot {F}loer homology and cabling}, Algebr. Geom.
  Topol. \textbf{5} (2005), 1197--1222,
  \href{http://arxiv.org/abs/math/0406402v2}{arXiv:math/0406402v2 [math.GT]}.
  \MR{2171808}

\bibitem[Hed10]{Hedden2010}
\bysame, \emph{Notions of positivity and the {O}zsv\'{a}th-{S}zab\'{o}
  concordance invariant}, J. Knot Theory Ramifications \textbf{19} (2010),
  no.~5, 617--629,
  \href{http://arxiv.org/abs/math/0509499v1}{arXiv:math/0509499v1 [math.GT]}.
  \MR{2646650}

\bibitem[HLR17]{Hom2017}
Jennifer Hom, Robert Lipschitz, and Daniel Ruberman, \emph{Thirty years of
  {F}loer theory for 3-manifolds}, 2017,
  \url{https://www.birs.ca/cmo-workshops/2017/17w5011/report17w5011.pdf}.

\bibitem[HTW98]{Hoste1998}
Jim Hoste, Morwen Thistlethwaite, and Jeff Weeks, \emph{The first 1,701,936
  knots}, The Mathematical Intelligencer \textbf{20} (1998), no.~4, 33--48.
  \MR{1646740}

\bibitem[HW18a]{Hedden2018}
Matthew Hedden and Liam Watson, \emph{On the geography and botany of knot
  {F}loer homology}, Selecta Math. (N.S.) \textbf{24} (2018), no.~2, 997--1037,
  \href{http://arxiv.org/abs/1404.6913v3}{arXiv:1404.6913v3 [math.GT]}.
  \MR{3782416}

\bibitem[HW18b]{HWbotanygeography}
\bysame, \emph{On the geography and botany of knot {F}loer homology}, Selecta
  Math. (N.S.) \textbf{24} (2018), no.~2, 997--1037. \MR{3782416}

\bibitem[Kob84]{kobayashi}
Tsuyoshi Kobayashi, \emph{Structures of the {H}aken manifolds with {H}eegaard
  splittings of genus two}, Osaka J. Math. \textbf{21} (1984), no.~2, 437--455.
  \MR{752472}

\bibitem[LV]{Lee}
Christine Ruey~Shan Lee and Faramarz Vafaee, \emph{On 3-braids and {L}-space
  knots}, unpublished manuscript.

\bibitem[Mor86]{MortonBraid}
H.~R. Morton, \emph{Seifert circles and knot polynomials}, Math. Proc.
  Cambridge Philos. Soc. \textbf{99} (1986), no.~1, 107--109. \MR{809504}

\bibitem[Mot16]{Motegi2016}
Kimihiko Motegi, \emph{L-space surgery and twisting operation}, Algebr. Geom.
  Topol. \textbf{16} (2016), no.~3, 1727--1772,
  \href{http://arxiv.org/abs/1405.6487v3}{arXiv:1405.6487v3 [math.GT]}.
  \MR{3523053}

\bibitem[Ni07]{Ni2007}
Yi~Ni, \emph{Knot {F}loer homology detects fibred knots}, Invent. Math.
  \textbf{170} (2007), no.~3, 577--608,
  \href{http://arxiv.org/abs/math/0607156v4}{arXiv:math/0607156v4 [math.GT]}.
  \MR{2357503}

\bibitem[OS04]{genusbounds}
Peter Ozsv\'{a}th and Zolt\'{a}n Szab\'{o}, \emph{Holomorphic disks and genus
  bounds}, Geom. Topol. \textbf{8} (2004), 311--334. \MR{2023281}

\bibitem[OS05]{Ozsvath2005}
\bysame, \emph{On knot {F}loer homology and lens space surgeries}, Topology
  \textbf{44} (2005), no.~6, 1281--1300,
  \href{http://arxiv.org/abs/math/0303017v2}{arXiv:math/0303017v2 [math.GT]}.
  \MR{2168576}

\bibitem[Pdt20]{pandas}
The {P}andas~development team, \emph{pandas-dev/pandas: Pandas}, 2020,
  \url{https://doi.org/10.5281/zenodo.3509134}.

\bibitem[Rol84]{rolfsen}
Dale Rolfsen, \emph{Rational surgery calculus: extension of {K}irby's theorem},
  Pacific J. Math. \textbf{110} (1984), no.~2, 377--386. \MR{726496}

\bibitem[Swe]{KLO}
Frank Swenton, \emph{{KLO} ({K}not-{L}ike {O}bjects)}, (29/07/2019),
  \url{http://KLO-Software.net}.

\bibitem[{The}20]{sagemath}
{The Sage Developers}, \emph{{S}age{M}ath, the {S}age {M}athematics {S}oftware
  {S}ystem ({V}ersion 9.1)}, 2020, {\tt https://www.sagemath.org}.

\bibitem[VRD09]{python}
Guido Van~Rossum and Fred~L. Drake, \emph{Python 3 reference manual},
  CreateSpace, Scotts Valley, CA, 2009.

\end{thebibliography}

\end{document}